\documentclass[11pt, a4paper]{amsart}

\RequirePackage[left=20mm,right=20mm,top=30mm,bottom=30mm]{geometry}

\usepackage{amsmath, amssymb, amsthm}
\usepackage{amscd}
\usepackage[all]{xy}

\usepackage{color}

\usepackage{tikz-cd}
 \usetikzlibrary{arrows, shapes, positioning,fit,calc, chains}
\usepackage[square,sort,comma,numbers]{natbib}
 \usepackage{url}
 \usepackage{hyperref}
 \usepackage[shortlabels]{enumitem} %to index the enumeration
\usepackage{mathabx} %package to use widecheck
\usepackage[skip=0pt]{caption}

\newtheorem{numberingthm}{Theorem}[section] % reset theorem numbering for each section
\theoremstyle{definition}
\newtheorem{Def}[numberingthm]{Definition}
\newtheorem{example}[numberingthm]{Example}

\theoremstyle{plain}
\newtheorem{Prop}[numberingthm]{Proposition}
\newtheorem{Theorem}[numberingthm]{Theorem}

\newtheorem{Cor}[numberingthm]{Corollary}
\newtheorem{Lemma}[numberingthm]{Lemma}

\newtheorem*{thmintroductionA}{Theorem~A}
\newtheorem*{thmintroductionB}{Theorem~B}

\newenvironment{Example}
{\pushQED{\qed}\example}
{\popQED\endexample}

\theoremstyle{remark}

\newtheorem{Remark}[numberingthm]{Remark}

\newcommand{\m}{\!\operatorname{-mod}}
\newcommand{\Ext}{\mathsf{Ext}}

\newcommand{\End}{\mathsf{End}}
\newcommand{\lerp}{{}^\perp}
\newcommand{\add}{\!\mathsf{add}}
\newcommand{\pdim}{\mathsf{pd}}
\newcommand{\Tr}{\mathsf{Tr}}

\newcommand{\domdim}{\mathsf{domdim}}
\newcommand{\Hom}{\mathsf{Hom}}
\newcommand{\ddim}{\mathsf{dim}}
\newcommand{\codim}{\mathsf{codim}}
\newcommand{\injdim}{\mathsf{id}}
\newcommand{\gldim}{\mathsf{gldim}}
\newcommand{\findim}{\mathsf{findim}}
\newcommand{\op}{\mathsf{op}}

\DeclareMathOperator*{\smod}{\mathsf{mod}-\!}
\DeclareMathOperator*{\lsmod}{\!-\mathsf{mod}}

\DeclareMathOperator*{\lusmod}{\!-\underline{\mathsf{mod}}}
\DeclareMathOperator*{\losmod}{\!-\overline{\mathsf{mod}}}
\newcommand{\CM}{\mathsf{CM}}

\DeclareMathOperator{\ldom}{\!-\mathsf{domdim}} %I removed the star because this operator should have nothing beneath
\DeclareMathOperator{\lcodom}{\!-\mathsf{codomdim}}

\begin{document}

\baselineskip=14pt

\title{A higher dimensional Auslander-Iyama-Solberg correspondence}

\author[T. Cruz]{Tiago Cruz}
\address[Tiago Cruz]{Institut f\"ur Algebra und Zahlentheorie,
Universit\"at Stuttgart, Germany }
\email{tiago.cruz@mathematik.uni-stuttgart.de}

\author[C. Psaroudakis]{Chrysostomos Psaroudakis}
\address[Chrysostomos Psaroudakis]{Department of Mathematics, Aristotle University of Thessaloniki,
   54124, Thessaloniki, Greece}
\email{chpsaroud@math.auth.gr}


\begin{abstract}
In this paper, we prove a higher dimensional version of Auslander-Iyama-Solberg correspondence. Iyama and Solberg  have shown a bijection between $n$-minimal Auslander-Gorenstein algebras and $n$-precluster tilting modules. If $A$ is an $n$-minimal Auslander-Gorenstein algebra, then the pair $(A,P)$ is a relative $(n+1)$-Auslander-Gorenstein pair in the sense of the authors, where $P$ is the minimal faithful projective-injective left $A$-module. 
 We establish a higher dimensional Auslander-Iyama-Solberg, where $P$ is replaced by any self-orthogonal module $Q$ having finite projective and injective dimension.
This new correspondence provides a bijection between relative Auslander--Gorenstein pairs and a new class of objects that generalise precluster tilting modules. This way, we obtain a new correspondence coming from the modular representation theory of general linear groups. 
\end{abstract}

\maketitle

\section{Introduction}

In its essence, the aim of representation theory is to classify all finite-dimensional algebras and their module categories. 
For instance, determining when an algebra is of finite representation type has been one of the most important and challenging problems in representation theory. In the early seventies, Auslander established a  bijection between algebras of finite representation type and algebras of global dimension at most two and of dominant dimension at least two. The latter algebras are known as Auslander algebras. Auslander's correspondence \cite{zbMATH03517355} brought in representation theory techniques and methods from category theory and homological algebra. It is certainly one of the most fundamental results in representation theory, opening doors to interactions between areas like representation theory, cluster theory, algebraic geometry and symplectic geometry.

Iyama proved in \cite{IyamaCorrespondence} a higher-dimensional analogue of Auslander's correspondence. This is a correspondence between $n$-Auslander algebras, i.e.\ algebras of global dimension at most $n+1$ and of dominant dimension at least $n+1$, and finite $n$-cluster tilting subcategories. This important correspondence, together with Iyama's work \cite{HigherARtheory} on $n$-Auslander-Reiten translation, $n$-Auslander-Reiten duality and $n$-almost split sequences, settled the foundational origins of what is now called higher-dimensional Auslander-Reiten theory. The first Gorenstein occurrence in these Auslander-type correspondences has already appeared in \cite{zbMATH00423522} by Auslander--Solberg, where in Auslander's original correspondence the global dimension has been replaced by the finitistic dimension of an Iwanaga--Gorenstein algebra. More recently, Iyama and Solberg \cite{IS} introduced the concept of $n$-precluster tilting subcategories and proved a bijective correspondence between this type of subcategories and $n$-minimal Auslander-Gorenstein algebras, i.e.\ Iwanaga--Gorenstein algebras with finitistic dimension at most $n+1$ and dominant dimension at least $n+1$. This important correspondence is a higher dimensional analog of Auslander--Solberg correspondence \cite{zbMATH00423522} and also a Gorenstein analog of Iyama's correspondence \cite{IyamaCorrespondence}. %for $n$-Auslander algebras.
It should be noted that all the above correspondences mentioned so far are special cases of the Morita--Tachikawa correspondence.

More recently, new Auslander-type correspondences started to emerge that are no longer restrictions of the Morita-Tachikawa correspondence. An example of this is the Li-Zhang correspondence \cite{LiZhang} that gives an Auslander-type correspondence for the so-called almost $n$-Auslander-Gorenstein algebras introduced in \cite{AT}. These algebras generalise in turn a collection of relative Auslander algebras studied by Iyama in \cite{zbMATH02111865}. For a subset of other generalisations of Auslander algebras we refer to \cite{Grevstad, zbMATH02111865, MR4392222, zbMATH07198564, MR4080882, chen2024iyamasolberg, chan2024dominant, MR4172701}.

In \cite{CrPs}, the authors introduced and studied 
the concept of relative Auslander--Gorenstein pairs. This consists of an Iwanaga-Gorenstein algebra $A$ together with a self-orthogonal $A$-module $Q$ such that the finitistic dimension of $A$ is smaller than or equal to the faithful dimension of $Q$. So the idea of relative Auslander--Gorenstein pairs is to replace in the definition of $n$-minimal Auslander-Gorenstein algebras the dominant dimension of the algebra with the faithful dimension of $Q$ in the sense of Buan-Solberg \cite{BuanSolberg}. The latter dimension fits into the general concept of relative (co)dominant dimension which has been extensively studied by the first author \citep{Cr2}. In \cite{CrPs}, relative Auslander--Gorenstein pairs are characterised in terms of the existence and uniqueness of tilting-cotiling modules having higher values of relative (co)dominant dimension with respect to the self-orthogonal module. The class of relative Auslander--Gorenstein pairs, respectively the class of relative Auslander pairs (i.e.\ relative Auslander--Gorenstein pairs with $A$ of finite global dimension), generalizes the class of $n$-minimal Auslander--Gorenstein algebras, respectively the class of $n$-Auslander algebras. %Pictorially we have the following inclusions between these classes of algebras.
To summarise, we have the following inclusions between these classes of algebras:
%\medskip

{
\centering
\begin{tabular}{ c c c } 
 %\hline
  & & \\[-0.5em]
 \big\{\text{Auslander algebras}\big\} & $\subseteq$ & \big\{ \text{Auslander-Solberg algebras}\big\} \\ 
  \rotatebox[origin=c]{270}{$\subseteq$} & & \rotatebox[origin=c]{270}{$\subseteq$} \\
 \big\{\text{$n$-Auslander algebras}\big\} & $\subseteq$ & \big\{\text{$n$-minimal Auslander-Gorenstein algebras}\big\} \\ 
\rotatebox[origin=c]{270}{$\subseteq$} & & \rotatebox[origin=c]{270}{$\subseteq$} \\
 \big\{\text{relative Auslander pairs}\big\} & $\subseteq$ & \big\{\text{relative Auslander-Gorenstein pairs}\big\} \\
 \rotatebox[origin=c]{90}{$\subseteq$} & & \rotatebox[origin=c]{90}{$\subseteq$} \\
\big\{\text{almost $n$-Auslander algebras}\big\} & $\subseteq$ & \big\{\text{almost  $n$-Auslander-Gorenstein algebras}\big\}. \\[-0.5em]
  & & \\
 %\hline
\end{tabular}
\par
}

%\medskip

Our aim in this paper is to prove an Auslander-type correspondence for the class of relative Auslander--Gorenstein pairs. 
In our setup, "higher dimensional" means that the correspondence has more homological variables than just the finitistic dimension. The reason for this is two-fold: In the case of $n$-minimal Auslander--Gorenstein algebras the module $Q$ is both projective and injective. However, here, this module can have large projective and injective dimensions. Secondly, to achieve our bijection we need to enlarge the class of precluster tilting objects used in \cite{IS, LiZhang}. Motivated by this, we introduce the concept of $(n, m, l)$-quasi-precluster tilting module (Definition~\ref{defnmlpreculuster}). The value $m$ comes from the injective dimension of $Q$, the value $l$ comes from the projective dimension of $Q$ and $n$ 
comes from the finitistic dimension of $A$. 
Moreover, $Q$ being an $(n, m, l)$-quasi-precluster tilting over its endomorphism algebra $\Lambda:=\End_A(Q)^{op}$ means that $Q$ over $\Lambda$ has a quasi-(co)generator structure as introduced in \cite{CruzMoritaTachikawa}.

The main result of the paper is stated below and it is what we propose to call the higher dimensional {Auslander--Iyama--Solberg} correspondence. 

\begin{thmintroductionA} \textnormal{(see Theorem~\ref{maintheorem})}
Let $k$ be a field. For every triple of non-negative integers $n$, $m$ and $l$ with $n\geq m+l+2$, there is a one-to-one correspondence between:
\begin{itemize}
\item pairs $(\Lambda, Q)$, where $\Lambda$ is a finite-dimensional algebra and $Q$ is an $(n, m, l)$-quasi-precluster tilting right $\Lambda$-module;

\item relative $n$-Auslander--Gorenstein pairs $(A, Q)$, where $Q$ is a self-orthogonal left $A$-module having projective dimension $l$ and injective dimension $m$.
\end{itemize}
The bijection is given by $(\Lambda, Q_\Lambda)\mapsto (\End_{\Lambda}(Q), Q)$ and $(A, {}_AQ)\mapsto (\End_A(Q)^{op}, Q)$.
\end{thmintroductionA}

Based on the above theorem we have in Figure \ref{HigherMTfigure} a hierarchy of correspondences in representation theory of finite-dimensional algebras. Here, a generalization is drawn as an arrow from the special case to the general case:

%\medskip

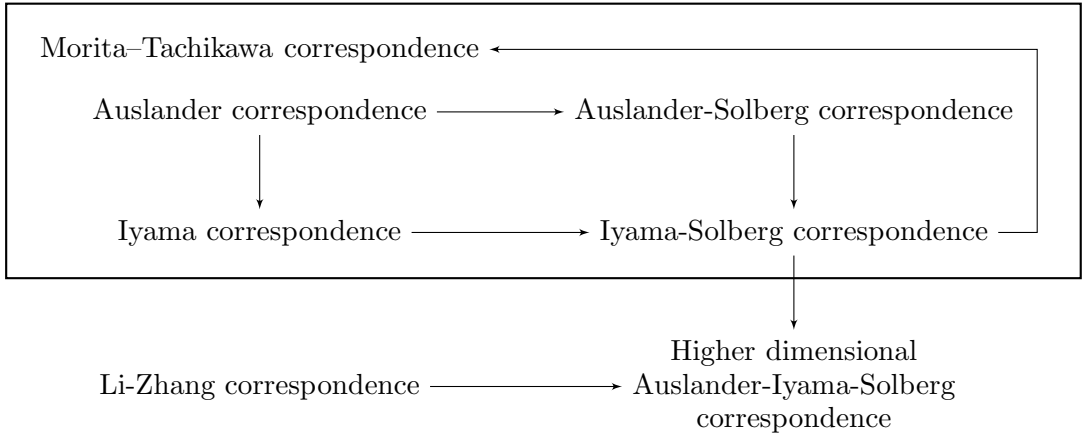
\begin{figure}[htb]
	\centering
	\begin{tikzpicture}
		[> = latex', auto,
		block/.style ={rectangle,draw=black, thick,
			align=flush center, rounded corners,
			minimum height=4em},
		FIT/.style args = {#1/#2}{draw=#1, thick, fit=#2,
			inner xsep=-2cm, inner ysep=1.5mm}
		]
		\matrix [column sep=10mm,row sep=10mm]
		{
			% row 1
			 \node[] (MT) {Morita--Tachikawa correspondence}; & \\[-2em]
		%	\node[text depth = 2.5cm] (main){Morita--Tachikawa correspondence}; 
			\node [] (AC) {Auslander correspondence}; &  \node [] (AS) {Auslander-Solberg correspondence};\\
			\node [] (IC) {Iyama correspondence}; &  \node [] (IS) {Iyama-Solberg correspondence};
			\\[0em]
			& \node [] (AIS) {Higher dimensional };\\[-3em]
			\node [] (LZ) {Li-Zhang correspondence};	& \node[] (AIStwo) {Auslander-Iyama-Solberg};\\[-3em]
			& \node[] (AISthree) {correspondence};\\
		};
		
		\node[FIT=red/(MT) (IS),draw=black,inner ysep=3mm, inner xsep=7mm, xshift=1em] {};

		\draw[->] (AC) -- (AS);
		\draw[->] (IC) -- (IS);
		\draw[->] (AC) -- (IC);
		\draw[->] (AS) -- (IS);
		\draw[->] (IS) -- (AIS);
		\draw[->] (LZ) -- (AIStwo);
		
		\draw [->] (IS) --  node[anchor=east, above=2pt ] {} ++(3.2,0) |- (MT);
		
   \end{tikzpicture}
	\caption{\label{HigherMTfigure}Specialisations of the Higher Morita-Tachikawa correspondence} 
\end{figure}
The significance of our new correspondence is that on one hand it unifies the previous correspondences but on the other, it provides new examples from modular representation theory. 
In particular, following \cite{CE}, this correspondence shows that all Temperley--Lieb algebras of infinite global dimension admit a quasi-precluster tilting module (see Example \ref{TLexemplo}).
It is well-known that the blocks of Schur algebras of $\operatorname{GL}_p$ of degree $p$ (over characteristic $p$) have finite representation-type (see \cite{zbMATH00098815}) and also that all blocks of Schur algebras of finite representation-type are higher Auslander algebras (see \cite{E2}).
Hence, the expectation is that there are way more Schur algebras (potentially from other classical groups) that should fit in our setup, and in particular we expect that many algebras arising as quotients of group algebras of symmetric groups possess quasi-precluster tilting modules.

The last contribution of our quasi-precluster tilting theory is on Cohen-Macaulay theory. For an Iwanaga-Gorenstein algebra $A$, we denote by $\CM(A)$
%\[
%\CM(A) = \{X\in A\lsmod \ | \ \Ext_A^i(X,A)=0 \ %\text{for all} \ i>0 \}
%\]
the category of  (left) Cohen-Macaulay $A$-modules. We write $\underline{\CM}(A)$ for the stable category of $\CM(A)$ and it is a well-known fact that $\CM(A)$ is a triangulated category. Given a relative $n$-Auslander--Gorenstein pair $(A, Q)$, the problem is to describe the category of Cohen-Macaulay $A$-modules. Based on relative homological algebra due to Auslander--Solberg, and following a similar strategy employed in \cite{IS} we establish our second main result. 
This consists of a description of Cohen-Macaulay $A$-modules via the quasi-precluster tilting module $Q$.

\begin{thmintroductionB} \textnormal{(see Theorem~\ref{thm8dot5})}	
Let $(A, Q)$ be a relative $n$-Auslander--Gorenstein pair with \mbox{$_{A}Q\in Q^\perp$} and $n\geq m+l+2$, where $m:=\injdim_{A} Q$ and $l:=\pdim_{A} Q$. Write $\Lambda=\End_A(Q)^{op}$. Then the functor $\Hom_\Lambda(-, Q)$ induces a duality between ${}^{\perp_m}  Q\cap Q^{\perp_{n-m-2}}$ and $\CM(A)$, where the former denotes the category
${ \{X\in \smod{\Lambda} \ | \ \Ext_\Lambda^i(X, Q)=0=\Ext_\Lambda^j(Q, X), 1\leq i\leq m, 1\leq j\leq n-m-2\}}$. In particular, there is a triangle equivalence$\colon$ 
\[
\xymatrix{
{}^{\perp_m}  \underline{Q\cap Q}^{\perp_{n-m-2}} \ar[r]^{ \ \ \ \ \simeq} &  \underline{\CM}(A^{op})}
\]
where ${}^{\perp_m}  \underline{Q\cap Q}^{\perp_{n-m-2}}$ is the category ${}^{\perp_m}  Q\cap Q^{\perp_{n-m-2}}$ modulo the ideal generated by $\add{Q}$.		
\end{thmintroductionB}

The paper is organised as follows. In Section~\ref{Preliminaries}, we collect results on homological dimensions and the terminology to be used throughout the paper. More precisely, the notion of relative dominant dimension and a relative version of Mueller's theorem are recalled (see Definition~\ref{denreldomdim} and Theorem~\ref{thm313}, respectively).  We recall in Theorem~\ref{higherMT} the higher Morita-Tachikawa correspondence, established in \cite{CruzMoritaTachikawa}, and we prove in Lemma~\ref{lemma2dot7} how we can use the relative dominant dimension to compute the injective or projective dimension of a module. In Section~\ref{section:quasipreclustertilting}, we introduce quasi-precluster tilting modules and we prove Theorem~A. 
From Lemma \ref{lemma5dot1} up to Proposition~\ref{Prop5dot4}, we describe the general properties that a self-orthogonal of a relative Auslander--Gorenstein pair possesses when viewed as a module over its endomorphism algebra. 
This leads us to Definition~\ref{defnmlpreculuster}, where the notion of $(n, m, l)$-quasi-	precluster tilting module is introduced. In Proposition~\ref{Prop5dot6}, we show how from an algebra with a quasi-cluster tilting module we obtain a relative Auslander-Gorenstein pair. The higher dimensional Auslander-Iyama-Solberg correspondence is proved in Theorem~\ref{maintheorem}. Section~\ref{schuralgebra} is devoted to examples. We illustrate Theorem~A by writing explicitly all computations for the Schur algebra $S(2,4)$. In Section~\ref{section: CM}, we prove Theorem~B. After recalling basic facts from the relative homological algebra of Auslander-Solberg \cite{zbMATH00423522}, we prove in Propositions~\ref{Prop8dot1} and~\ref{Prop8dot3} that $Q$ is a relative tilting and cotilting object in certain exact categories related with $\CM(A)$. Using these facts, we prove in Theorem~\ref{thm8dot5} the above equivalence for Cohen-Macaulay $A$-modules. 

\section{Preliminaries}\label{Preliminaries}

Throughout the paper, let $A$ be a finite-dimensional algebra over a field $k$ and $d$ a natural number. We write $A\lsmod$ to denote the category of finitely generated left $A$-modules and $\smod{A}$ to denote the category of finitely generated right $A$-modules. Let  $Q$ be a module in $A\lsmod $. We write $\lerp Q$ to denote $\{Y\in A\lsmod \ | \ \Ext_A^i(Y, Q)=0, \, \forall i>0 \}$, and similarly the notation $Q^{\perp}$ is self explanatory. By $A^{op}$ we mean the opposite algebra of $A$ and $D$ denotes the standard duality ${D=\Hom_{k}(-,k)\colon A\lsmod \to A^{op}\lsmod}$.  Given $X\in A\lsmod$ the additive closure of $X$ is denoted by $\add_A X$ (or just $\add X$, or even $\add X_A$ when $X\in \smod{A}$). Given $M$ in $A\lsmod$, we will denote by $\pdim_{A} M$ (resp. $\injdim_A M$) the projective (resp. injective) dimension of $M$ over $A$. By $\gldim A$ we denote the global dimension of $A$. Given a non-negative integer $m$ and $M\in A\lsmod$, $\Omega_A^mM$ (resp. $\Omega_A^{-m} M$ denotes the $m$-syzygy (resp. $m$-cosyzygy) of $M$.  We use $\nu_A$ to denote the Nakayama functor $D\Hom_A(-, A)$. 

 The \textbf{finitistic dimension} of $A$, denoted as $\findim A$, is defined by $$\findim A=\sup\{ \pdim_A X\colon X\in A\m, \pdim_A X<+\infty\}.$$ 
A finite-dimensional algebra $A$ is an $n$-\textbf{Iwanaga--Gorenstein} whenever both modules ${}_AA$ and $A_A$ have injective dimension at most $n$. 
For Iwanaga--Gorenstein algebras the dimensions $\injdim_A A=\injdim A_A$ coincide (see \citep[Lemma 6.9]{AuslanderReiten}). Moreover, Iwanaga--Gorenstein algebras have finite finitistic dimension and it follows from \citep[VI.5, Lemma 5.5]{zbMATH00707210} that for those we have $\findim A=\injdim_A A$.
 Over an $n$-Iwanaga--Gorenstein algebra the modules in $\lerp A$ are known as \textbf{Gorenstein-projective $A$-modules}. In this setup, these modules are also known as \textbf{Cohen-Macaulay $A$-modules}.

We start by recalling the concept of relative (co)dominant dimension presented in \citep{Cr2}.

\begin{Def}
\label{denreldomdim}
Let $A$ be a finite-dimensional algebra and $Q, X$ in $A\lsmod$.
\begin{enumerate}[\normalfont(i)]
\item We say that the {\bf relative dominant dimension} of $X$ with respect to $Q$ is greater than or equal to $d$, denoted by $Q\ldom_AX\geq d$, if there is an exact sequence 
\[
0\to X\to Q_1\to Q_2\to \cdots \to Q_d, 
\]
with $Q_i\in \add{Q}$, such that the next sequence is exact
\[
 \ \ \ \ \xymatrix{
\Hom_A(Q_d, Q) \ar[r]^{}  & \cdots \ar[r]^{} & \Hom_A(Q_2, Q) \ar[r]^{} & \Hom_A(Q_1, Q) \ar[r]^{} & \Hom_A(X, Q) \ar[r]^{} &  0.
 } 
\]
Otherwise, we say that the {\bf relative dominant dimension} of $X$ with respect to $Q$ is zero.  The {\bf relative dominant dimension} of $X$ with respect to $Q$ is infinite if it is greater than or equal to any natural number. 

\item We say that the {\bf relative codominant dimension} of $X$ with respect to $Q$ is greater than or equal to $d$, denoted by $Q\lcodom_AX\geq d$, if there is an exact sequence $$Q_d\to \cdots\to Q_2\to Q_1\to X \to 0,$$ with $Q_i\in \add{Q}$, such that the next sequence is exact
\[
\ \ \ \ \xymatrix{
\Hom_A(Q, Q_d) \ar[r]^{}  & \cdots \ar[r]^{} & \Hom_A(Q, Q_2) \ar[r]^{} & \Hom_A(Q, Q_1) \ar[r]^{} & \Hom_A(Q, X) \ar[r]^{} &  0.
 } 
\]
Otherwise, we say that the {\bf relative codominant dimension} of $X$ with respect to $Q$ is zero.
\end{enumerate}
\end{Def}

For $Q$ being a projective-injective module, the relative (respectively, co)dominant dimension of $M$ with respect to $Q$ is exactly the (respectively, co)dominant dimension of $M$. It is clear from the definition that $Q\ldom_A X=DQ\lcodom_{A^{op}} DX$ for every $X, Q\in A\lsmod$.
Observe that Gorenstein projective $A$-modules over an Iwanaga--Gorenstein algebra have infinite relative dominant dimension with respect to $A$ (see for example \cite{AuslanderReiten}).

Various useful facts of relative (co)dominant dimension from \cite{Cr2} are recalled later in the paper when they are used. What we definitely need in the sequel is the following relative version of Mueller theorem, for more details see for example \cite[Theorem~3.1.3 and Theorem~3.1.4]{Cr2}.

\begin{Theorem}\label{thm313}
Let $A$ be a finite-dimensional algebra, $Q$ in $A\lsmod$ and $\Lambda=\End_A(Q)^{op}$. 
\begin{enumerate}
\item[\textnormal{$(\alpha)$}] For $M$ in $A\lsmod$ the following statements are equivalent:
\begin{enumerate}[\normalfont(i)]
\item The evaluation map $M\to \Hom_{\Lambda}(\Hom_{A}(M, Q), Q)$ is an isomorphism. 

\item $Q\ldom_A M\geq 2$.
\end{enumerate}
\item[\textnormal{$(\beta)$}] The following statements are equivalent: 
\begin{enumerate}[\normalfont(i)]
\item $Q\ldom_A{A}\geq n$. 

\item $\Ext_\Lambda^i(Q,Q)=0$ for $1\leq i\leq n-2$ and $A\cong \End_{\Lambda}(Q)$, i.e.\ the module $Q$ has the double centralizer property.\end{enumerate}
\end{enumerate}
\end{Theorem}

In view of Theorem \ref{thm313}, the relative dominant dimension of $A$ can be computed in terms of exact sequences $0\rightarrow A\rightarrow Q_1\rightarrow \cdots$ with $Q_i\in \add Q$ so that the next sequence under $\Hom_A(-, Q)$ is a minimal projective resolution of $Q\cong \Hom_A(A, Q)$ over $\Lambda:=\End_A(Q)^{op}$. Indeed, pick a minimal projective resolution of $\Hom_A(A, Q)$, then applying $\Hom_\Lambda(-, Q)$ yields an exact sequence of the desired form if and only if $Q\ldom_A A\geq n$ thanks to Theorem \ref{thm313}. This fact is also valid for computing $Q\ldom_A X$ for an arbitrary $X\in A\m$, but for that one needs to use the general case of relative Mueller's theorem developed in \cite[Theorem~3.1.3 and Theorem~3.1.4]{Cr2}.

The following notion was introduced in \cite{CrPs}.

\begin{Def}\label{defAusGorpair}
Let $A$ be a finite-dimensional algebra and $Q\in A\lsmod$ such that $Q\in Q^{\perp}$. The pair $(A, Q)$ is called a {\bf relative $n$-Auslander--Gorenstein pair} if the algebra $A$ is an $n$-Iwanaga--Gorenstein algebra satisfying the following
\[ 
\injdim_A A\leq n\leq Q\ldom_A A.
\]

The pair $(A, Q)$ is called a {\bf relative $n$-Auslander pair} if it is a relative $n$-Auslander--Gorenstein pair and $A$ has finite global dimension.
\end{Def}

The following quasi-(co)generation notion was introduced by the first author in \cite{CruzMoritaTachikawa} as a first step towards understanding the behaviour of objects with the property $M\ldom_A A=+\infty$.

\begin{Def}
Let $A$ be a finite-dimensional algebra over a field and let $M\in {A}\lsmod$.
\begin{enumerate}[\normalfont(i)]
\item We say that $M$ is an {\bf $n$-quasi-generator} of $A\lsmod$ if $n$ is the minimal integer such that there exists an exact sequence 
\[
\xymatrix{
0 \ar[r]^{} & A \ar[r] & M_0\ar[r] & M_1\ar[r] & \cdots \ar[r]^{} & M_n \ar[r]^{} &  0} 
\]
which remains exact under $\Hom_A(-,M)$ and every $M_i\in \add_AM$.

\item We say that $M$ is an {\bf $n$-quasi-cogenerator} of $A\lsmod$ if $n$ is the minimal integer such that there exists an exact sequence 
\[
\xymatrix{
0 \ar[r]^{} & M_n \ar[r] & \cdots \ar[r] & M_1\ar[r] & M_0 \ar[r]^{} & DA \ar[r]^{} &  0} 
\]
which remains exact under $\Hom_A(M,-)$ and every $M_i\in \add_AM$.
\end{enumerate}
\end{Def}

By setting $n=0$, we recover the notions of generator and cogenerator, respectively. The importance of the quasi-(co)generator notion is justified by the higher Morita-Tachikawa correspondence presented below, which is due to the first author (see \cite{CruzMoritaTachikawa}).

\begin{Theorem}
\label{higherMT}
There is a bijection between:
\begin{itemize}
\item pairs $(B, M)$ where $B$ is a finite-dimensional algebra and $M$ is an $n$-quasi-generator of $\smod{B}$;

\item pairs $(A, M)$ where $A$ is a finite-dimensional algebra with %relative dominant dimension 
$M\ldom_A A\geq 2$, $\pdim _AM=n$ and $\Ext_A^{i>0}(M, M)=0$.
\end{itemize}
The bijection is given by $(B, M)\mapsto (\End_{B}(M), M)$ and $(A, {_AM})\mapsto (\End_A(M)^{op}, M)$.
\end{Theorem}

For every $A$-module $M$, $\add M$ is functorially finite (see for instance \cite{zbMATH03749214}), so we can consider relative dimensions with respect to $\add Q$.

\begin{Def}
Let $A$ be a finite-dimensional algebra over a field and let $Q\in {A}\lsmod$. 
\begin{enumerate}[\normalfont(i)]
\item A left $A$-module $M$ has {\bf relative $\add{Q}$-dimension} at most $m$, denoted by $\ddim_{ \, \add{Q}}(M)\leq m$, if there
exists an exact sequence 
\[
\xymatrix{
0 \ar[r]^{} & Q_m' \ar[r] & \cdots \ar[r]^{} & Q_1' \ar[r]^{} & Q_0' \ar[r]^{} & M \ar[r]^{} &  0} 
\]
which remains exact under $\Hom_A(Q,-)$ and $Q_i'\in \add_{A}Q$.

\item A left $A$-module $M$ has {\bf relative $\add{Q}$-codimension} at most $l$, denoted by $\codim_{ \, \add{Q}}(M)\leq l$, if there
exists an exact sequence
\[
\xymatrix{
0 \ar[r]^{} & M \ar[r] &  Q_0'' \ar[r] & Q_1'' \ar[r] & \cdots \ar[r]^{} & Q_l'' \ar[r]^{} & 0} 
\]
which remains exact under $\Hom_A(-,Q)$ and $Q_i''\in \add_{A}Q$.
\end{enumerate}
\end{Def}

In particular, $\ddim_{ \, \add{Q}}(DA)=n$  exactly when $Q$ is an $n$-quasi-cogenerator of $A\lsmod$.
Of course, by fixing $Q=A$ in the previous definition, we recover the absolute projective dimension and absolute injective dimension. We also need in the sequel the following standard homological fact.

\begin{Lemma}
\label{inductionprojdimension}
Let $0\to X_{n-l}\to Q_{n-l-1}\to \cdots \to Q_1\to Q_0\to X_0\to 0$ be an exact sequence in ${A}\lsmod$ with $Q_i\in \add{Q}$. Assume that $\pdim {_AX_0}\leq n$ and $\pdim {_AQ}=l$ with $n>l$. Then $\pdim {_AX_{n-l}}\leq l$.
\end{Lemma}

For modules with higher relative dominant dimension with respect to $Q$ we can relate their absolute injective dimension using relative $\add Q$-dimensions.

\begin{Lemma}\label{lemma2dot7}
Let $Q, X\in A\lsmod$ with $Q\ldom A\geq 2$. Then, the following assertions hold.
\begin{enumerate}[\normalfont(i)]
\item If $X\in Q^\perp$ and $Q\lcodom_A X\geq 2$, then $\injdim_A X=\dim_{\, \add Q_\Lambda} D\Hom_A(Q, X)$.
		
\item If $X\in \lerp Q$ and $Q\ldom_A X\geq 2$, then $\pdim_A X=\codim_{ \, \add{Q}}D\Hom_A(X, Q)$.
	\end{enumerate} 
\end{Lemma}
\begin{proof}
Suppose that $\injdim_A X<+\infty$.  Let 
\begin{align}
	\xymatrix{0 \ar[r] &X \ar[r] &I_0 \ar[r] & \cdots \ar[r] & I_d \ar[r] & 0}
\end{align}be the minimal injective resolution 
of $X$. Thanks to $\Ext_A^{i>0}(Q, X)=0$ applying $\Hom_A(Q, -)$ to the minimal injective resolution of $X$  yields the exact sequence 
\begin{align}
	\xymatrix{0 \ar[r] & \Hom_A(Q, X)\ar[r] & \Hom_A(Q, I_0) \ar[r] & \cdots \ar[r] & \Hom_A(Q, I_d)\ar[r] & 0}. \label{eq23}
\end{align} Applying $\Hom_{\Lambda}(-, DQ)$ we get the commutative diagram
\begin{equation*}
	\xymatrix@C=0.42cm@R=0.4cm{0 \ar[r] & \Hom_{\Lambda}(\Hom_A(Q, I_d), DQ) \ar[r]  & \cdots \ar[r] & \Hom_{\Lambda}(\Hom_A(Q, I_0), DQ) \ar[r] & \Hom_{\Lambda}(\Hom_A(Q, X),DQ) \ar[r] & 0 \\
0 \ar[r]	& \Hom_A(I_d, DA) \ar[r]   \ar^{\cong}[u]& \cdots \ar[r] & \Hom_A(I_0, DA) \ar[r]  \ar^{\cong}[u]& \Hom_A(X, DA) \ar[r]  \ar^{\cong}[u]& 0 \\
0 \ar[r] & DI_d \ar[r] \ar^{\cong}[u] & \cdots \ar[r] & DI_0  \ar^{\cong}[u] \ar[r] & DX \ar[r]  \ar^{\cong}[u] & 0}
\end{equation*} Since the bottom row is exact, then so (\ref{eq23}) remains exact under $\Hom_{\Lambda}(-, DQ)$.  Hence,
\begin{equation}
	\xymatrix{0 \ar[r] & D\Hom_A(Q, I_d) \ar[r] & \cdots \ar[r]& D\Hom_A(Q, I_0) \ar[r] & D\Hom_A(Q, X) \ar[r] & 0}
\end{equation} is exact and it remains exact under $\Hom_{\Lambda}(Q, -)$. Moreover, $D\Hom_A(Q, I_i)\in \add Q_\Lambda$. So, this shows that $\dim_{\, \add Q_\Lambda} D\Hom_A(Q, X)\leq \injdim_A X$. 

Conversely, assume that $\dim_{\, \add Q_\Lambda} D\Hom_A(Q, X)\leq d$ for some non-negative number $d$. Hence, $\pdim \Hom_\Lambda(Q, D\Hom_A(Q, X))_{A}\leq d$. Moreover, $$\Hom_\Lambda(Q, D\Hom_A(Q, X))\cong \Hom_\Lambda(\Hom_A(Q, X), DQ)\cong \Hom_\Lambda(\Hom_A(DX, DQ), DQ)\cong DX$$ since $Q\lcodom_A X\geq 2$. It follows that $\pdim DX_A\leq d$ and $\injdim_A X\leq d$. So (1) follows. Dually, (2) holds.
\end{proof}

We write 
$\tau_{n}=\tau \Omega_A^{n-1}\colon {A}\lusmod \to {A}\losmod$ and $\tau^{-}_{n}=\tau^{-} \Omega_A^{-(n-1)}\colon {A}\losmod \to {A}\lusmod$ for the $n$-Auslander-Reiten translation \cite{HigherARtheory}, where ${A}\lusmod$ denotes the stable category of $A\lsmod$ modulo projective modules and ${A}\losmod$ denotes the stable category of $A\lsmod$ module injective modules.

We recall the following result of Li and Zhang \cite{LiZhang}, see also the proof of Proposition~3.9 in \cite{ChenKoenig}.

\begin{Lemma} \textnormal{(\cite[Lemma~4.5]{LiZhang})}
\label{lemLiZhang}
Let $A$ be a finite-dimensional algebra. Let $M$ be an $A$-module such that $\Ext^i_A(M,M)=0$ for all $1\leq i\leq n-1$. Consider a minimal projective resolution $\cdots \to P_1\to P_0\to M\to 0$ of $M$. Then we have the following exact sequence:
\[
\xymatrix@C=0.7cm{
0 \ar[r]^{} & \Hom_A(M,M) \ar[r] & \Hom_A(P_0, M) \ar[r] & \cdots \ar[r]^{} & \Hom_A(P_n, M) \ar[r]^{} & D\Hom_A(M, \tau_n M) \ar[r]^{} & 0}. 
\]
\end{Lemma}

\section{Quasi-precluster tilting modules}
\label{section:quasipreclustertilting}

In this section, we introduce the concept of quasi-precluster tilting modules and we establish the main result that relative Auslander--Gorenstein pairs can be characterised by the existence of these objects. 

We start by showing that for certain natural numbers $d$ the objects $\nu \Omega^d Q$ and $\tau_{d-1} Q$ have finite relative $\add Q$-codimension when $Q$ is a quasi-generator arising from a relative Auslander--Gorenstein pair.

\begin{Lemma}\label{lemma5dot1}
	Let $(A, Q)$ be a relative $n$-Auslander--Gorenstein pair with $\pdim_A Q=l$, $\injdim_A Q=m$ and $_{A}Q\in Q^\perp$. Write $\Lambda=\End_A(Q)^{op}.$ Assume that $n\geq m+l+2$, then there exists
	an exact sequence 
	\[
	\xymatrix@C=0.5cm{
		0 \ar[r]^{} &\tau_{n-l-1}^-Q \ar[r] & Q_0 \ar[r] & Q_1 \ar[r] & \cdots \ar[r]  &Q_l \ar[r] & 0 } 
	\] which remains exact under $\Hom_{\Lambda}(-, Q)$ with $Q_i\in \add Q_\Lambda$.
\end{Lemma}
\begin{proof}
	By definition, $Q\ldom_A A\geq n\geq \injdim {}_A A$. Hence, we also have $Q\lcodom_A DA\geq n\geq \pdim_A DA$, see for example \citep[Corollary 3.1.5]{Cr2}. So, there exists an exact sequence
	\begin{equation}
			\xymatrix@C=0.5cm{
			0 \ar[r]^{} &X_n \ar[r] & Q_{n-1} \ar[r] & Q_{n-2} \ar[r] & \cdots \ar[r]  &Q_0 \ar[r] &  DA\ar[r] &  0 }  \label{eq7}
	\end{equation} which remains exact under $\Hom_A(Q, -)$ with $Q_i\in \add_A Q$ for $0\leq i\leq n-1$. Here, $X_0$ denotes ${}_A DA$ and $X_{i+1}$ is defined recursively as the kernel of the map $Q_i\rightarrow X_{i}$ for $0\leq i\leq n-1$.

 We aim to compute $\tau^-_{n-l-1}Q\cong \Tr D \Omega^{-(n-l-2)} Q$ using the exact sequence (\ref{eq7}). We can assume without loss of generality that (\ref{eq7}) is a minimal exact sequence in the sense that the induced exact sequence 
 	\begin{equation}
 	\xymatrix@C=0.5cm{
 		 \Hom_A(Q, Q_{n-1}) \ar[r] & \Hom_A(Q, Q_{n-2}) \ar[r] & \cdots \ar[r] & \Hom_A(Q, Q_0) \ar[r] & \Hom_A(Q, DA)\cong DQ \ar[r] & 0} \label{eq5'}
 \end{equation} 
 is the beginning of the minimal projective resolution of $DQ$ as $\Lambda$-module. So applying $D\Hom_A(Q, -)$ to the exact sequence (\ref{eq7}) we obtain the exact sequence
  	\begin{equation}
 	\xymatrix@C=0.5cm{
 		0 \ar[r] & Q \ar[r] & D\Hom_A(Q, Q_{0}) \ar[r] & \cdots \ar[r] & D\Hom_A(Q, Q_{n-l-3}) \ar[r] & D\Hom_A(Q, X_{n-l-2}) \ar[r] & 0}
 	\label{eq8}
 \end{equation} if $n-l-2>0$. Hence, $D\Hom_A(Q, X_{n-l-2})\cong \Omega^{-(n-l-2)}Q$  if $n-l-2>0$. Otherwise, we have $\Omega^{-(n-l-2)} Q\cong Q\cong D\Hom_A(Q, X_{n-l-2})$, and so we conclude that the following identification holds for all cases 
 $$\Omega^{-(n-l-2)}Q\cong D\Hom_A(Q, X_{n-l-2}).$$

On the other hand, using dimension shifting together with $Q\in Q^\perp$ on exact sequence (\ref{eq7}) we obtain 
\begin{equation}
	\Ext_A^i(X_{n-l-2}, Q)\cong \Ext_A^{i+n-l-2}(X_0, Q)=\Ext_A^{i+n-l-2}(DA, Q)=0, \quad  \forall i>0,
	\end{equation}since $n-l-2\geq m=\injdim_AQ$. Moreover, $X_i\in \lerp Q$ for every $i=n-l-2, \ldots, n$. So applying $\Hom_A(-, Q)$ to the exact sequence $0\rightarrow X_{n-l}\rightarrow Q_{n-l-1}\rightarrow Q_{n-l-2}\rightarrow X_{n-l-2}\rightarrow 0$ yields the exact sequence
	\begin{equation}
	\xymatrix@C=0.45cm{
		0\ar[r] &\Hom_A(X_{n-l-2}, Q)  \ar[r] & \Hom_A(Q_{n-l-2}, Q) \ar[r] & \Hom_A(Q_{n-l-1}, Q) \ar[r] & \Hom_A(X_{n-l}, Q) \ar[r] & 0 
	}. \label{eq10}
\end{equation} In particular, this infers that $Q\ldom_A X_{n-l}\geq 2$.  
Therefore, by applying $\Hom_\Lambda(-, \Lambda)$ to (\ref{eq5'}) we get the commutative diagram with exact rows
\begin{equation*}
	\begin{tikzcd}[column sep=0.75cm]
		\Hom_\Lambda(\Hom_A(Q, X_{n-l-2}), \Lambda) \ar[r, hookrightarrow] & \Hom_\Lambda(\Hom_A(Q, Q_{n-l-2}), \Lambda)  \ar[r] \ar[ld, "\cong"]& \Hom_\Lambda(\Hom_A(Q, Q_{n-l-1}), \Lambda) \ar[ld, "\cong"]  \\
		 \Hom_A(Q_{n-l-2}, Q) \ar[r] & \Hom_A(Q_{n-l-1}, Q) \ar[r, twoheadrightarrow] & \Hom_A(X_{n-l}, Q)
	\end{tikzcd},
\end{equation*} where the bijectivity of the vertical maps follows from projectivization together with $\Lambda=\End_A(Q)^{op}$. Using the above commutative diagram, we then obtain 
$$\Hom_A(X_{n-l}, Q)\cong \Tr \Hom_A(Q, X_{n-l-2})\cong \Tr DD \Hom_A(Q, X_{n-l-2}) \cong \Tr D \Omega^{-(n-l-2)}Q =\tau^-_{n-l-1}Q.$$

 Since $\pdim_A DA=\pdim_A X_0\leq n$ and $\pdim_A Q\leq l$ we obtain that $\pdim_A X_{n-l}\leq l$ using  Lemma \ref{inductionprojdimension}. Let 
 \begin{equation}
 		\xymatrix@C=0.5cm{
 		0\ar[r] & P_{l}\ar[r] & \cdots \ar[r] & P_0 \ar[r] & X_{n-l} \ar[r] & 0
 	}\label{eq12}
 \end{equation} be a minimal projective resolution of $X_{n-l}$. Since $X_{n-l}\in \lerp Q$ applying $\Hom_A(-, Q)$ yields the exact sequence
 \begin{equation}
	\xymatrix@C=0.5cm{
		0\ar[r] & \Hom_A(X_{n-l}, Q)\ar[r] & \Hom_A(P_0, Q) \ar[r] & \cdots \ar[r] & \Hom_A(P_l, Q) \ar[r] & 0.
	}\label{eq13}
\end{equation} Observe that  $\Hom_A(P_i, Q)\in \add Q_\Lambda$ for all $i=0, \ldots l$.  It remains to check that the exact sequence (\ref{eq13}) remains exact under $\Hom_{\Lambda}(-, Q)$.
This follows immediately from the following commutative diagram and the exactness of (\ref{eq12}):
\begin{equation*}
	\begin{tikzcd}
		\Hom_\Lambda(\Hom_A(P_l, Q), Q) \ar[r] & \cdots \ar[r] & \Hom_\Lambda(\Hom_A(P_0, Q), Q) \ar[r] & \Hom_{\Lambda}(\Hom_A(X_{n-l}, Q), Q) \\
		P_l \ar[r, hookrightarrow] \ar[u, "\cong"] & \cdots \ar[r] & P_0 \ar[r, twoheadrightarrow] \ar[u, "\cong"] & X_{n-l} \ar[u, "\cong"]
	\end{tikzcd}
\end{equation*}
Here, the vertical maps are isomorphisms by Theorem \ref{thm313} since $Q\ldom_A P_i\geq Q\ldom_A A\geq 2$ and $Q\ldom_A X_{n-l}\geq 2$. So, our desired exact sequence is the exact sequence (\ref{eq13}) replacing $\Hom_A(X_{n-l}, Q)$ with $\tau_{n-l-1}^-Q $.
\end{proof}

\begin{Remark}
	It follows from Lemma \ref{lemma5dot1} that if $Q$ is a projective module with injective dimension $m$ and $(A, Q)$ is a relative $n$-Auslander--Gorenstein pair with $n\geq m+2$, then $\tau^-_{n-1}Q\in \add Q_{\Lambda}$.
\end{Remark}

\begin{Remark}\label{Rmk7dot1}%
Assume that $(A, Q)$ is a relative $n$-Auslander--Gorenstein pair with $\pdim_A Q=l$, $\injdim_A Q=m$ and $_{A}Q\in Q^\perp$. Suppose that $n\geq m+l+2$. Then, it follows from Lemma \ref{lemma5dot3} that $\tau_{n-m-1}Q\cong D\tau_{n-m-1}^-(DQ) \cong D\Hom_A(\tilde{X}_{n-m}, DQ)\cong D\Hom_A(Q, D\tilde{X}_{n-m})$, where $\tilde{X}_{n-m}$ fits into an exact sequence $0\rightarrow \tilde{X}_{n-m}\rightarrow \tilde{Q}_{n-m-1}\rightarrow \cdots \rightarrow \tilde{Q}_0\rightarrow DA\rightarrow 0$ which remains exact under $\Hom_A(DQ, -)$ with $\tilde{Q}_i\in \add DQ$.
\end{Remark}

The dual version of Lemma \ref{lemma5dot1} is the following:

\begin{Lemma}\label{lemma5dot3}
		Let $(A, Q)$ be a relative $n$-Auslander--Gorenstein pair with $\pdim_A Q=l$, $\injdim_A Q=m$ and $_{A}Q\in Q^\perp$. Write $\Lambda=\End_A(Q)^{op}.$ Assume that $n\geq m+l+2$, then there exists
	an exact sequence 
	\[
	\xymatrix@C=0.5cm{
		0 \ar[r] & Q_m \ar[r] & Q_{m-1} \ar[r] & \cdots \ar[r]  &Q_0 \ar[r]  & \tau_{n-m-1}Q \ar[r]& 0 } 
	\] which remains exact under $\Hom_{\Lambda}(Q,-)$ with $Q_i\in \add Q_{\Lambda}$.
\end{Lemma}
\begin{proof}
Observe that $Q\ldom_A A=DQ\ldom_{A^{op}} A$, see for example \citep[Corollary 3.1.5]{Cr2}. Hence \cite[Remark~4.9]{CrPs} yields that $(A^{op}, DQ)$ is a relative $n$-Auslander-Gorenstein pair with $\pdim_{A^{op}} DQ=m$, $\injdim_{A^{op}} DQ=l$ and $n\geq m+l+2$. By Lemma \ref{lemma5dot1}, there exists an exact sequence
		\begin{equation}
	\xymatrix@C=0.5cm{
		0 \ar[r]^{} &\tau_{n-m-1}^-DQ \ar[r] & Q_0' \ar[r] & Q_1' \ar[r] & \cdots \ar[r]  &Q_m' \ar[r] & 0 } \label{eq14}\end{equation}
	which remains exact under $\Hom_{\Lambda^{op}}(-, DQ)$ with $Q_i'\in \add_\Lambda DQ$. Observe that 
 $$D\tau^-_{n-m-1}DQ=D\Tr D \Omega^{-(n-m-2)}DQ\cong D\Tr \Omega^{n-m-2}DDQ\cong D\Tr \Omega^{n-m-2}Q=\tau_{n-m-1}Q.$$ 
 Hence, our desired exact sequence is obtained by applying $D$ to (\ref{eq14}).
\end{proof}

In general, we have the following identification between the Nakayama functor and the syzygies of $Q$ for relative Auslander-Gorenstein pairs $(A, Q)$. 

\begin{Lemma}
	Let $(A, Q)$ be a relative $n$-Auslander--Gorenstein pair with $\pdim_A Q\leq d$, $\injdim_A Q\leq n-d$ and $_{A}Q\in Q^\perp$ for some natural number $d\in \{1, \ldots, n-1\}$. Write $\Lambda=\End_A(Q)^{op}$. Then the following assertions hold.
	\begin{enumerate}[\normalfont(i)]
		\item If $d\geq 2$, then $\add \Omega^d(Q)\oplus \Lambda_\Lambda =\add \nu_\Lambda^-\Omega^{-(n-d)}(Q) \oplus \Lambda_\Lambda$;
		\item If $n-d\geq 2$, then $\add \nu_\Lambda \Omega^d(Q)\oplus D\Lambda = \add D\Lambda \oplus \Omega^{-(n-d)}(Q)$.
	\end{enumerate}
\end{Lemma}
\begin{proof}
	Since $d>0$ there are exact sequences
		\begin{align}
		\xymatrix{
			0 \ar[r] & Y \ar[r] & Q_{n-d-1}'\ar[r] & \cdots \ar[r] & Q_0'\ar[r] & DA\ar[r] & 0 
		} \\ \xymatrix{ 0 \ar[r] & A \ar[r] & Q_0 \ar[r] & \cdots \ar[r] & Q_{d-1} \ar[r] & X \ar[r] & 0
	}
	\end{align} with $Q_i, Q_i'\in \add_AQ$ 
so that the induced exact sequences
\begin{align}
	\xymatrix{
	0 \ar[r] & \Hom_A(Q, Y) \ar[r]& \Hom_A(Q, Q_{n-d-1}') \ar[r] & \cdots \ar[r] & \Hom_A(Q, Q_0') \ar[r] & DQ\ar[r] & 0
	} \\ \xymatrix{ 0 \ar[r] & \Hom_A(X, Q) \ar[r] & \Hom_A(Q_{d-1}, Q) \ar[r] & \cdots \ar[r] & \Hom_A(Q_0, Q)\ar[r] & Q\ar[r] & 0
	}
\end{align}
are the beginning of minimal projective resolutions of $DQ$ and $Q$, respectively.
Hence, $\Hom_A(Q, Y)\cong \Omega^{n-d}(DQ)$ and $\Omega^d(Q)\cong \Hom_A(X, Q)$. By \cite[Corollary~6.5]{CrPs}, we have $\add X\oplus Q=\add Y\oplus Q$.

We will now prove (i). Assume that $d\geq 2$. Then, $Q\lcodom_A Q\oplus X\geq 2$. In such a case, we have the following isomorphisms as right $\Lambda$-modules
\begin{align*}
	\Lambda_\Lambda\oplus \Omega^d(Q)&\cong \Hom_A(Q, Q)\oplus \Hom_A(X, Q) \\
&\cong \Hom_A(Q\oplus X, Q) \\&\cong \Hom_{\Lambda}(\Hom_A(Q, Q\oplus X), \Hom_A(Q, Q)) 
\end{align*}
and 
\begin{align*}
\Hom_\Lambda(\Hom_A(Q, Q\oplus Y), \Hom_A(Q, Q))&\cong \Hom_\Lambda(\Lambda\oplus\Omega^{n-d}(DQ), \Lambda)\\ &\cong \Lambda \oplus \Hom_\Lambda(D\Omega^{-(n-d)}(Q), \Lambda)\\&\cong \Lambda \oplus \nu_\Lambda^-(\Omega^{-(n-d)}(Q)). 
\end{align*} Hence, it follows that $\add \Lambda \oplus \nu_\Lambda^-\Omega^{-(n-d)}(Q)=\add \Lambda \oplus \Omega^d(Q)$.

Assume now instead that $n-d\geq 2$. In such a case $Q\ldom_A Q\oplus X\geq 2$, and so we have the following isomorphisms as right $\Lambda$-modules
\begin{align*}
	D\Lambda\oplus \nu_\Lambda \Omega^d(Q)&\cong D \Lambda\oplus D\Hom_{\Lambda}(\Hom_A(X, Q), \Hom_A(Q, Q))\\&\cong D\Lambda\oplus D\Hom_{\Lambda}(\Hom_A(DQ, DX), \Hom_A(DQ, DQ))\\&\cong D\Hom_A(Q, Q)\oplus D\Hom_A(DX, DQ)\\ &\cong D\Hom_A(Q, Q)\oplus D\Hom_A(Q, X)\\&\cong D\Hom_A(Q, Q\oplus X)
\end{align*}
and \begin{align*}
	D\Hom_A(Q, Q\oplus Y)\cong D\Hom_A(Q, Y)\oplus D\Lambda \cong D\Omega^{n-d}(DQ)\oplus D\Lambda\cong \Omega^{-(n-d)}(Q)\oplus D\Lambda.
\end{align*} So (ii) follows and this completes the proof.
\end{proof}

\begin{Lemma}\label{lemma7dot3}
		Let $(A, Q)$ be a relative $n$-Auslander--Gorenstein pair with $\pdim_A Q\leq d \leq n-2$, $\injdim_A Q\leq n-d$ and $_{A}Q\in Q^\perp$ for some natural number $d\geq 2$. Write $\Lambda=\End_A(Q)^{op}$. Then, there exists an exact sequence 
		\begin{equation}
			\xymatrix{
			0 \ar[r] & Q_{n-d} \ar[r] & \cdots \ar[r] & Q_0 \ar[r] & \nu_\Lambda \Omega^d(Q) \ar[r] & 0
		}
		\end{equation} which remains exact under $\Hom_{\Lambda}(Q, -)$ with $Q_i\in \add Q$.
\end{Lemma}
\begin{proof}
By \cite[Proposition~6.4]{CrPs}, without loss of generality, there exists an exact sequence 
\begin{equation}
		\xymatrix{
		0 \ar[r] & A \ar[r] & Q_0 \ar[r] & \cdots \ar[r] & Q_{d-1} \ar[r] & X \ar[r] & 0 }
\end{equation}which remains exact under $\Hom_A(-, Q)$ with the following extra properties:
\begin{itemize}
	\item $Q_i\in \add_A Q$;
	\item $\Hom_A(X, Q)\simeq \Omega^d(Q)$;
	\item $Q\oplus X$ is a $d$-tilting $(n-d)$-cotilting module;
	\item $Q\ldom_A X\geq n-d\geq 2$;
	\item $Q\lcodom_A X\geq d\geq 2$.
\end{itemize}
In particular, $X\in \lerp Q\cap Q^\perp$ and $\injdim_A X\leq n-d$. By Lemma \ref{lemma2dot7}, $\dim_{\, \add Q_\Lambda} D\Hom_A(Q, X)\leq n-d.$ Hence, it is enough to prove that $D\Hom_A(Q, X)\cong \nu_\Lambda \Omega^d(Q)$. This holds true since $Q\ldom_A X\geq 2$. Indeed,
\begin{align*}
	\nu_\Lambda \Omega^d(Q)&\cong D\Hom_{\Lambda}(\Omega^d(Q), \Lambda) \\ &\cong D\Hom_{\Lambda}(\Hom_A(X, Q), \Lambda) \\ & \cong D\Hom_{\Lambda}(\Hom_A(X, Q), \Hom_A(Q, Q))
	\\&\cong D\Hom_{\Lambda}(\Hom_{A}(DQ, DX), \Hom_{A}(DQ, DQ)) \\ & \cong D\Hom_A(DX, DQ)\cong D\Hom_A(Q, X)
\end{align*}
and this completes the proof.
\end{proof}

\begin{Remark}
		Let $(A, Q)$ be a relative $n$-Auslander--Gorenstein pair with $_{A}Q\in Q^\perp$. Assume that $n\geq m+l+2$, where $m:=\injdim_A Q$ and $l:=\pdim_A Q$.  Then, in particular, $n\geq m+2$ and following the same argument as in the proof of Lemma \ref{lemma7dot3} we get $\nu_\Lambda \Omega^{n-m}(Q)\cong D\Hom_A(Q, X)$ (with $d:=n-m$). By Remark \ref{Rmk7dot1}, we infer that $\nu_\Lambda \Omega^{n-m}(Q)\cong \tau_{n-m-1}(Q)$.
\end{Remark}

Combining the previous techniques together with the higher Morita--Tachikawa correspondence (Theorem~\ref{higherMT}) we obtain the following properties for the self-orthogonal module associated with the relative Auslander--Gorenstein pair.

\begin{Prop}\label{Prop5dot4}Let $n, m, l$ be natural numbers so that $n\geq m+l+2$.
		Let $(A, Q)$ be a relative $n$-Auslander--Gorenstein pair with $\pdim_A Q=l$, $\injdim_A Q=m$ and $_{A}Q\in Q^\perp$. Define $\Lambda=\End_A(Q)^{op}$.
		Then the following assertions hold:
		\begin{enumerate}[\normalfont(i)]
			\item $\Ext_{\Lambda}^i(Q, Q)=0$ for $1\leq i\leq n-2$.
			
			\item $Q$ is an $l$-quasi-generator of $\smod{\Lambda}$. 
			
			\item $Q$ is an $m$-quasi-cogenerator of $\smod{\Lambda}$.
			
			\item There exists an exact sequence 
			\[
			\xymatrix{
				0 \ar[r]^{} & Q_m' \ar[r] & \cdots \ar[r]^{} & Q_1' \ar[r]^{} & Q_0' \ar[r]^{} & \tau_{n-m-1}Q \ar[r]^{} &  0} 
			\]
			which remains exact under $\Hom_{\Lambda}(Q,-)$ and $Q_i\in \add Q_{\Lambda}$.
			
			\item There exists an exact sequence 
			\[
			\xymatrix{
				0 \ar[r]^{} & \tau^{-}_{n-l-1}Q \ar[r] &  Q_0'' \ar[r] & Q_1'' \ar[r] & \cdots \ar[r]^{} & Q_l'' \ar[r]^{} & 0} 
			\]
			which remains exact under $\Hom_{\Lambda}(-,Q)$ and $Q_i''\in \add Q_{\Lambda}$.
		\end{enumerate}
\end{Prop}
\begin{proof}
By assumption, $Q\ldom_A A\geq n\geq \injdim_AA$ and $A$ is a Gorenstein algebra. By the higher Morita-Tachikawa correspondence (see \citep[Theorem 3.5 ]{CruzMoritaTachikawa}), (ii) and (iii) are satisfied. By Theorem \ref{thm313}, (i) is satisfied. By Lemmas \ref{lemma5dot1} and  \ref{lemma5dot3}, (v) and (iv) are satisfied, respectively.
\end{proof}

Motivated by the above result we introduce the following concept. 

\begin{Def}
\label{defnmlpreculuster}
Let $\Lambda$ be a finite-dimensional algebra and let $Q\in \smod{\Lambda}$. We say that $Q$ is an \textbf{$(n, m, l)$-quasi-precluster tilting module} if $Q$ satisfies the Conditions (i)-(v) of Proposition \ref{Prop5dot4}.
\end{Def}

 We provide examples of quasi-precluster tilting modules in Section~\ref{schuralgebra}.

\begin{Prop}\label{Prop5dot6}
	Let $\Lambda$ be a finite-dimensional algebra and let $Q\in \smod{\Lambda}$. We write $A=\End_{\Lambda}(Q)$ and let $n$ be a natural number satisfying $n\geq \sup\{l, m+2\}$. Assume the following conditions:
	\begin{enumerate}[\normalfont(i)]
		\item $\Ext_{\Lambda}^i(Q, Q)=0$ for $1\leq i\leq n-2$.
		
		\item $Q$ is an $l$-quasi-generator of $\smod{\Lambda}$. 
		
		\item $Q$ is an $m$-quasi-cogenerator over $\smod{\Lambda}$.
		
		\item $\ddim_{ \, \add{Q}}\tau_{n-m-1}(Q)\leq m$.

\item $\codim_{ \, \add{Q}}\tau_i^{-}(Q)<\infty$ for some $1\leq i\leq n-1$.
\end{enumerate}
Then  $(A, Q)$ is a relative $n$-Auslander-Gorenstein pair.
\end{Prop}
\begin{proof}
By the higher Morita-Tachikawa correspondence, see Theorem~\ref{higherMT}, we obtain that $\pdim _AQ=l$ since $Q$ satisfies condition (ii). Similarly, condition (iii) implies that $\injdim _AQ=m$ and $Q\ldom_A A\geq 2$. Then using (i) and \cite[Theorem~3.1.4]{Cr2} we deduce that $Q\ldom_A A\geq n$.
		
		It remains to show that $\injdim{_AA}\leq n$ and $\injdim{A_A}\leq n$. 
		By (i) we have that $\Ext^i_{\Lambda}(Q, Q)=0$ for $i=1,\ldots, n-2$ and in particular for $i=1,\ldots, n-m-2$. Since $n-m-2\geq 0$, by Lemma~\ref{lemLiZhang} we have an exact sequence 
		\begin{equation}
			\label{firstexactseq}
			\xymatrix@C=0.4cm{
				0 \ar[r]^{} & \Hom_{\Lambda}(Q, {}_AQ) \ar[r] & \Hom_{\Lambda}(P_0, Q) \ar[r] & \cdots \ar[r]^{} & \Hom_{\Lambda}(P_{n-m-1}, Q) \ar[r]^{} & D\Hom_{\Lambda}(Q, \tau_{n-m-1} Q) \ar[r]^{} & 0.} 
		\end{equation}
		By (iv), there exists an exact sequence 
		\[
		\xymatrix{
			0 \ar[r]^{} & \Hom_{\Lambda}(Q, Q'_m) \ar[r] & \cdots \ar[r] & \Hom_{\Lambda}(Q, Q_0') \ar[r] & \Hom_{\Lambda}(Q, \tau_{n-m-1} Q) \ar[r]^{} & 0} 
		\]
		and therefore $\pdim_A \Hom_{\Lambda}(Q, \tau_{n-m-1} Q)\leq m$. This implies that $\injdim_A D\Hom_{\Lambda}(Q, \tau_{n-m-1} Q)\leq m$. 
		
		Observe that given a short exact sequence of $\Lambda$-modules
		\[
		0\to X_1\to X_2\to X_3\to 0
		\]
		with $\injdim X_2\leq m$ and $\injdim X_3\leq m+k$, then $\injdim X_1\leq m+k+1$ where $k\geq 0$. From the exact sequence $(\ref{firstexactseq})$, and since we also have that $\Hom_{\Lambda}(P_i, Q)$ lies in $\add_AQ$ and $\injdim_A \Hom_{\Lambda}(P_i, Q) \leq m$, we infer that $\injdim {}_A A\leq m+ n - m = n$ as wanted. 
		
		Consider now a minimal projective resolution $\cdots \to P_1'\to P_0'\to DQ\to 0$. Since $\Ext^i_{\Lambda}(DQ, DQ)=0$ for $1\leq i\leq n-2$, we get the following exact sequence:
		\begin{equation}
			\label{secondexactseqt}
			\xymatrix@C=0.30cm{
				0 \ar[r]^{} & \Hom_{\Lambda}(DQ, DQ) \ar[r] & \Hom_{\Lambda}(P_0', DQ) \ar[r] & \cdots \ar[r] & \Hom_{\Lambda}(P'_{i}, DQ) \ar[r] & D\Hom_{\Lambda}(DQ, \tau_{i}(DQ)) \ar[r]^{} & 0.}
		\end{equation}
	Since $\pdim_AQ=l$ it follows that $\injdim DQ_A=l$ and therefore $\injdim\Hom_{\Lambda}(P_i', DQ)_A\leq l$.
	Now, Condition (v) implies that $\pdim_A\Hom_{\Lambda}(\tau_i^{-}(Q), Q)$ is finite and so $\injdim D\Hom_{\Lambda}(\tau_i^{-}(Q), Q)_A<\infty$. We also have
	\begin{align*}
		D\tau^{-}_{i}(Q) &= D\tau^{-}\Omega_{\Lambda}^{-i}(Q)  \\ 
		&= D\Tr D\Omega_{\Lambda}^{-i}(Q)  \\ 
		&\cong \tau\Omega_{\Lambda}^{i}(DQ) \\ &=\tau_i(DQ) 
	\end{align*}and therefore we obtain that 
\begin{align*}
	D\Hom_{\Lambda}(DQ, \tau_i(DQ))\cong D\Hom_{\Lambda}(D\tau_i(DQ), DDQ)\cong D\Hom_{\Lambda}(\tau^-_i(Q), Q).
\end{align*}Hence, we have that $\injdim_A D\Hom_{\Lambda}(DQ, \tau_i(DQ))$ is finite and so $$A_A\cong \Hom_{\Lambda}({}_AQ, Q)\cong \Hom_{\Lambda}(DQ, (DQ)_A)$$ has finite injective dimension using (\ref{secondexactseqt}). This shows that $A$ is Iwanaga--Gorenstein, and  therefore {$\injdim A_A=\injdim {}_A A\leq n$} by \cite[Lemma~6.9]{AuslanderReiten}. We conclude that the pair $(A, Q)$ is a relative $n$-Auslander--Gorenstein pair.\end{proof}

We are now ready to prove the higher dimensional  Auslander-Iyama-Solberg correspondence as stated in the introduction.

\begin{Theorem}\label{maintheorem}
	Let $k$ be a field. For every triple of non-negative integers $n$, $m$ and $l$ with $n\geq m+l+2$, there is a one-to-one correspondence between:
	\begin{itemize}
		\item pairs $(\Lambda, Q)$, where $\Lambda$ is a finite-dimensional algebra and $Q$ is an $(n, m, l)$-precluster tilting right $\Lambda$-module;
		\item relative $n$-Auslander--Gorenstein pairs $(A, Q)$, where $Q$ is a self-orthogonal left $A$-module having projective dimension $l$ and injective dimension $m$.
	\end{itemize}
The bijection is given by $(\Lambda, Q_\Lambda)\mapsto (\End_{\Lambda}(Q), Q)$ and $(A, {}_AQ)\mapsto (\End_A(Q)^{op}, Q)$.
\end{Theorem}
\begin{proof}
	It follows from Propositions \ref{Prop5dot4} and \ref{Prop5dot6} (setting $i=n-l-1$ in (v) of Proposition \ref{Prop5dot6}). The fact that these assignments form a one-to-one correspondence follows from these assignments being specialisations of the higher Morita-Tachikawa correspondence (see \citep{CruzMoritaTachikawa}).
\end{proof}

\begin{Remark}Let $n$ be a natural number greater than one. Let $Q\in A\lsmod$ with the property that $Q\ldom_A A$ is greater than $n-1$. 
	From the exact sequence (\ref{secondexactseqt}) and the higher Morita-Tachikawa correspondence it is clear that $\injdim A_A$ is finite if and only if $\Hom_{\Lambda}(DQ, \tau_{i}(DQ))$ has finite projective dimension over $A$, where $\Lambda$ denotes the endomorphism algebra $\End_A(Q)^{op}$. However, the beginning of a projective resolution of $\Hom_{\Lambda}(DQ, \tau_{i}(DQ))$ is only induced by an $\add_A DQ$-resolution of $\tau_i(DQ)$ in case $DQ\lcodom_\Lambda \tau_{i}(DQ)$ is greater than or equal to 2 for some $i=1, \ldots, n-1$. However, such a condition is only immediate when $DQ$ is a generator, or when $Q$ is a cogenerator of $\smod{\Lambda}$.
\end{Remark}

\section{Examples}
\label{schuralgebra}

 In this section, we illustrate Theorem~\ref{maintheorem} with explicit examples. We start with an important class of examples from the modular representation theory of general linear groups.

\begin{Example}\label{TLexemplo}
Let $K$ be a field of characteristic two and let $S_{K, q}(n, d)$ be the $q$-Schur algebra corresponding to the Dipper-Donkin quantum group $q$-$\operatorname{Gl}_n$ for some $0\neq q\in K$ with degree $d$ (see for example \cite{Do2}.)
    
    Temperley--Lieb algebras either have a quasi-hereditary structure or have infinite global dimension (we refer to \cite{CE} for details). Those that have infinite global dimension are of the form $TL_{K, d}(0)$ and these are exactly those that arise as $\End_{S_{K, q}(2, d)}((K^2)^{\otimes d})$ with $d$ an even natural number and $K$ having quantum characteristic two (see for example \citep[Corollary 6.8]{CE}). 
    
    Assume that $q^2=1$.
    By \citep[Theorem 3.10]{P}, the global dimension of $S_{K, q}(2, 2m)$ is exactly $2m$. (Observe that the original result assumes that $K$ is an algebraically closed field, but this also holds over any field because the algebraic closure of a field $k$ is faithfully flat as a module over $k$, faithful flat extensions commute with Ext-groups and it is compatible with the quasi-hereditary structure of $S_{k, q}(2, 2m)$).
    
    By \citep[Theorem 5.8]{CE}, it follows that $(S_{{K, q}(2, 2m)}((K^2)^{\otimes 2m}), (K^2)^{\otimes 2m})$ is a relative $2m$-Auslander pair for every $m\in \mathbb{N}$.
    
    So, Theorem \ref{maintheorem} restricts to a correspondence with $q$-Schur algebras together with the tensor power as a summand of the characteristic tilting module on one side and Temperley--Lieb algebras on the second side. In particular, this shows that all Temperley--Lieb algebras of infinite global dimension admit a quasi-precluster tilting module.
\end{Example}

We illustrate Proposition \ref{Prop5dot4}, using the Schur algebra $S(2, 4)$ over an algebraically closed field $k$ with characteristic two. Let $A$ be the basic algebra isomorphic to $S(2, 4).$ Hence, $A$ is the algebra $kQ/I$, where $Q$ is the quiver
\[
\xymatrix{
 3 \ar@/_0.5pc/[r]_{\alpha} & 5 \ar@/_0.5pc/[l]_{\alpha_1} \ar@/^0.5pc/[r]^{\beta} &  4 \ar@/^0.5pc/[l]^{\beta_1}
}
\]
and $I$ is the ideal generated by the relations
\[
\alpha_1\alpha = 0 = \beta\beta_1 \ \ \text{and} \ \ \beta_1\beta\alpha = 0 = \alpha_1\beta_1\beta,
\] (see for example \citep[5.6]{E2} or \cite[Example 5.10]{CE}.)
The indecomposable projective modules are
\[
P(3):=\begin{tikzcd}[every arrow/.append style={dash}, column sep={0.4em},
		row sep={0.5em}]
 3 \arrow[d] \\ 5 \arrow[d] \\4
	\end{tikzcd}
\qquad
P(5):=\begin{tikzcd}[every arrow/.append style={dash}, column sep={0.4em},
		row sep={0.5em}]
 & 5 \arrow[ld] \arrow[rd] &\\ 3 \arrow[d]& & 4 \arrow[d]\\ 5 \arrow[d] & & 5 \\ 4 & &
	\end{tikzcd}
\qquad
\text{and}
\qquad
P(4):=\begin{tikzcd}[every arrow/.append style={dash}, column sep={0.4em},
		row sep={0.5em}]
 4 \arrow[d]\\ 5 \arrow[d] \\ 3 \arrow[d] \\ 5 \arrow[d] \\ 4
	\end{tikzcd}.
\]
The basic version of $(k^2)^{\otimes 4}$ is the left $A$-module
\[
Q\:= P(4)\oplus 
\begin{tikzcd}[every arrow/.append style={dash}, column sep={0.4em},
		row sep={0.5em}]
 5 \arrow[d] \\ 4 \arrow[d] \\ 5
	\end{tikzcd}.
\]
and notice that $\pdim_AQ = 1 = \injdim_AQ$.  By \cite{CE}, we have that
\[
Q\ldom S(2,4) = 4 = \gldim S(2,4)
\]
and therefore we are in the case of $l=m=1$ and $n=4$. $A$ has a quasi-hereditary structure and $T=3\oplus Q$ is the characteristic tilting module of $A$ and, in particular, it is the unique 2-tilting 2-cotilting module of $A$ that completes $Q$ in the sense of \citep[Theorem 6.6]{CrPs}.

Let $\Lambda$ be the endomorphism algebra of $Q$. So the quiver of $\Lambda$ is given by
\begin{equation}
\xymatrix{
 1 \ar@/^0.5pc/[r]^{\alpha} & 2 \ar@/^0.5pc/[l]^{\beta} \ar@(ur,dr)[]^{t}
} \label{ringeldualstwofour}
\end{equation}
bound by the relations
\[
\alpha\beta = \beta t = t\alpha = t^2 = 0.
\]
The indecomposable projective modules are
\[
P(1)=\begin{tikzcd}[every arrow/.append style={dash}, column sep={0.4em},
		row sep={0.5em}]
 1 \arrow[d] \\ 2 \arrow[d] \\ 1
	\end{tikzcd}.
\qquad
\text{and}
\qquad
P(2):=\begin{tikzcd}[every arrow/.append style={dash}, column sep={0.4em},
		row sep={0.5em}]
 & 2 \arrow[dl] \arrow[dr]& \\ 1 & & 2
	\end{tikzcd}. 
\]Denote by $f_1$ and $f_2$ the primitive idempotents associated with the vertices $1$ and $2$, respectively.
We want to compute 
\[
DQ \cong \Hom_{A}(Q, DA) \cong \Hom_{R}\big(\Hom_A(T, Q), \Hom_A(T, DA) \big)
\] as a left $\Lambda$-module. The Ringel dual of $A$, $R=\End_A(T)^{op}$, is then the bound quiver algebra of the quiver
\[
\xymatrix{
 1 \ar@/^0.5pc/[r]^{\alpha} & 2 \ar@/^0.5pc/[l]^{\beta} \ar@/^0.5pc/[r]^{\gamma} &  3 \ar@/^0.5pc/[l]^{\theta}
}
\]
with relations
\[
\gamma \alpha = \beta\theta = \alpha\beta = \gamma\theta = 0.
\]
In this notation, the module $\Hom_A(T, Q)$ is the projective module $R(e_1+e_2)$, where $e_i$ is the primitive idempotent associated with the vertex $i$ for $i\in\{1, 2, 3\}$. Moreover, we have that $\Hom_A(T, DA)$ is the tilting module $T_R(1)\oplus T_R(2)\oplus T_R(3)$, where
\[
T_R(1) = 1,
\qquad
T_R(2)=\begin{tikzcd}[every arrow/.append style={dash}, column sep={0.4em},
		row sep={0.5em}]
 1 \arrow[d] \\ 2 \arrow[d] \\ 1
	\end{tikzcd}
\qquad
\text{and}
\qquad
T_R(3) = \begin{tikzcd}[every arrow/.append style={dash}, column sep={0.4em},
		row sep={0.5em}]
 & 2 \arrow[dl] \arrow[d]& \\ 1 & 3 \arrow[d] & 1 \arrow[dl]\\ & 2 &
	\end{tikzcd}.
\]
Since $\Lambda\cong \Lambda^{\op}$ we have
\[
Q \cong DDQ \cong (e_1+e_2)\Hom_A(T, DA)\cong   1\oplus \begin{tikzcd}[every arrow/.append style={dash}, column sep={0.4em},
		row sep={0.5em}]
 1 \arrow[d] \\ 2 \arrow[d] \\ 1
	\end{tikzcd} \oplus \begin{tikzcd}[every arrow/.append style={dash}, column sep={0.4em},
		row sep={0.5em}]
 & 2 \arrow[d] \arrow[dl]&  1 \arrow[dl] \\ 1 & 2 & 
	\end{tikzcd}=:Q_1\oplus Q_2\oplus Q_3.
\]
We claim that $\add{Q}_\Lambda$ satisfies conditions (i)-(v) of Proposition~\ref{Prop5dot4}. 

Condition (i) is clear from the relative dominant dimension $Q\ldom A$ together with 
Theorem~\ref{thm313}. Condition (ii) is valid since we can consider the exact sequence

\[ 0 \rightarrow
\begin{tikzcd}[every arrow/.append style={dash}, column sep={0.4em},
		row sep={0.5em}]
 1 \arrow[d] \\ 2 \arrow[d] \\ 1
	\end{tikzcd} \oplus  \begin{tikzcd}[every arrow/.append style={dash}, column sep={0.4em},
		row sep={0.5em}]
 & 2 \arrow[dl] \arrow[dr] & \\ 1 & & 2
	\end{tikzcd} \to \begin{tikzcd}[every arrow/.append style={dash}, column sep={0.4em},
		row sep={0.5em}]
 1 \arrow[d] \\ 2 \arrow[d] \\ 1
	\end{tikzcd} \oplus \begin{tikzcd}[every arrow/.append style={dash}, column sep={0.4em},
		row sep={0.5em}]
 & 2 \arrow[d] \arrow[dl]&  1 \arrow[dl] \\ 1 & 2 & 
	\end{tikzcd} \to 1 \to 0.
\]
Indeed, this exact sequence remains exact under $\Hom_{\Lambda}(-, Q)$ since $Q_2$ is injective and the functor ${\Hom_{\Lambda}(-, Q_1\oplus Q_3)}$ induces the exact sequence
\[
0\rightarrow \Hom_\Lambda(Q_1, Q_1\oplus Q_3)\rightarrow \Hom_\Lambda(Q_3, Q_1\oplus Q_3)\rightarrow \Hom_\Lambda(P(2), Q_1\oplus Q_3)\rightarrow \Ext_\Lambda^1(Q_1, Q_1\oplus Q_3)\rightarrow 0
\]
and the vector space dimension of $\Ext_A^1(Q_1, Q_1\oplus Q_3)$ is therefore equal to 2-4+2=0. 
This means that $Q$ is an $1$-quasi-generator of $\smod{\Lambda}$.

Since $\Lambda$ has a simple preserving duality and $Q$ is self-dual under this duality, therefore Condition (iii) is verified with the exact sequence 
\[
0 \rightarrow Q_1\rightarrow Q_3\oplus Q_2\rightarrow D\Lambda\rightarrow 0.
\]

For Condition (iv) we have:
\begin{eqnarray}
\tau_2(Q) &=& \tau_{4-1-1}(Q) \nonumber \\
          &=& \tau \Omega(Q) \nonumber \\
          &=& \tau \Omega\big(Q_1\oplus Q_2\oplus Q_3\big)  \nonumber \\
&=& \tau \Omega(1) \oplus \tau\Omega P(1)\oplus \tau\Omega Q_3  \nonumber \\
&=& \tau\left(\begin{tikzcd}[every arrow/.append style={dash}, column sep={0.4em},
		row sep={0.5em}]
  2 \arrow[d] \\ 1
	\end{tikzcd}\right) \oplus \tau\left(\begin{tikzcd}[every arrow/.append style={dash}, column sep={0.4em},
		row sep={0.5em}]
  2 \arrow[d] \\ 1
	\end{tikzcd}\right). \nonumber
\end{eqnarray}

Fix $M:=\begin{tikzcd}[every arrow/.append style={dash}, column sep={0.4em},
		row sep={0.5em}]
  2 \arrow[d] \\ 1
	\end{tikzcd}.$ We claim that $\tau M= \begin{tikzcd}[every arrow/.append style={dash}, column sep={0.4em},
		row sep={0.5em}]
  1 \arrow[d] \\ 2
	\end{tikzcd}.$

We have the exact sequence

\[
\begin{tikzcd}
    f_2\Lambda \arrow[rr] \arrow[rd, twoheadrightarrow] & & f_2\Lambda \arrow[r] & M \arrow[r] & 0\\
   & 2 \arrow[ur, hookrightarrow]& 
\end{tikzcd}
\]

and applying the functor $\Hom_{\Lambda}(-, \Lambda)$ we obtain that

\[
0\rightarrow \Hom_\Lambda\left(  \begin{tikzcd}[every arrow/.append style={dash}, column sep={0.4em},
		row sep={0.5em}]
2 \arrow[d] \\1
	\end{tikzcd} , P(1)\oplus P(2)\right) \rightarrow \Hom_\Lambda(f_2\Lambda, P(1)\oplus P(2))\rightarrow \Hom_\Lambda(f_2\Lambda, P(1)\oplus P(2))
\]

which turns out to give the next exact sequence of left $\Lambda$-modules
\[
0\rightarrow 1\oplus 2 \rightarrow \Lambda f_2 \rightarrow \Lambda f_2 \rightarrow \begin{tikzcd}[every arrow/.append style={dash}, column sep={0.4em},
		row sep={0.5em}]
2 \arrow[d] \\1
	\end{tikzcd} \rightarrow 0.
\]

We infer that $\tau\left(\begin{tikzcd}[every arrow/.append style={dash}, column sep={0.4em},
		row sep={0.5em}]
2 \arrow[d] \\1
	\end{tikzcd} \right)= \begin{tikzcd}[every arrow/.append style={dash}, column sep={0.4em},
		row sep={0.5em}]
1 \arrow[d] \\2
	\end{tikzcd} $ as a right $\Lambda$-module.

The exact sequence giving condition (iv) is therefore 
\[
0\to Q_1\oplus Q_1 \to Q_2\oplus Q_2\rightarrow \tau_2 Q\cong \tau M\oplus \tau M\rightarrow 0,
\] since $\Ext^1_\Lambda(Q, Q_1)=0$.
We compute
\[
\tau_2^{-}(Q)  = \tau^{-}\Omega^{-1}\big(Q_1\oplus Q_2\oplus Q_3  \big) = \tau^{-}\left(\begin{tikzcd}[every arrow/.append style={dash}, column sep={0.4em},
		row sep={0.5em}]
1 \arrow[d] \\2
	\end{tikzcd}\oplus \begin{tikzcd}[every arrow/.append style={dash}, column sep={0.4em},
		row sep={0.5em}]
1 \arrow[d] \\2
	\end{tikzcd}\right) = \begin{tikzcd}[every arrow/.append style={dash}, column sep={0.4em},
		row sep={0.5em}]
2 \arrow[d] \\1
	\end{tikzcd} \oplus \begin{tikzcd}[every arrow/.append style={dash}, column sep={0.4em},
		row sep={0.5em}]
2 \arrow[d] \\1
	\end{tikzcd}. 
\]
So the exact sequence giving condition (v) is
\[
0\to \begin{tikzcd}[every arrow/.append style={dash}, column sep={0.4em},
		row sep={0.5em}]
2 \arrow[d] \\1
	\end{tikzcd}\oplus \begin{tikzcd}[every arrow/.append style={dash}, column sep={0.4em},
		row sep={0.5em}]
2 \arrow[d] \\1
	\end{tikzcd} \to P(1)\oplus P(1) \to 1\to 0.
\]
We claim that $\tau_3(Q)$ lies in $\add{Q}$. Indeed, observe that
\[
\tau_3(Q) = \tau\Omega^2(Q) = \tau \Omega \left(\begin{tikzcd}[every arrow/.append style={dash}, column sep={0.4em},
		row sep={0.5em}]
2 \arrow[d] \\1
	\end{tikzcd} \oplus \begin{tikzcd}[every arrow/.append style={dash}, column sep={0.4em},
		row sep={0.5em}]
2 \arrow[d] \\1
	\end{tikzcd} \right) = \tau(2\oplus 2)= Q_3\oplus Q_3
\]
which belongs to $\add{Q}$. So, this means that the dual of Condition (v) of Proposition \ref{Prop5dot6} also holds in this case. However, this is the best one can hope for since we now claim that $\tau_4(Q)$ cannot fit in any finite $\add{Q}$ resolution. In fact, 
\[
\tau_4(Q) = \tau\Omega^3(Q) = \tau\Omega (2\oplus 2) = \tau(1\oplus 2 \oplus 1 \oplus 2) = \tau(1)^2\oplus \tau(2)^2=(P(2)/1)^2\oplus Q_3^2,
\]but $P(2)/1$ cannot be a submodule of a module in $\add Q_1\oplus Q_2\oplus Q_3.$

\begin{Example}
    The condition $n\geq m+l+2$ in Theorem \ref{maintheorem} seems to be of importance even when $m=l=0$. 

    Indeed, let $A$ be the bound quiver algebra
%    \[
%\begin{tikzcd}
%1 \arrow[out=90,in=180,loop, "\alpha", swap] \arrow[r, "\gamma"]& 2  \arrow[out=0,in=90, loop, "\beta", swap]
%\end{tikzcd}
%\] 
\[
\xymatrix{
 1 \ar@(dl,ul)[]^{\alpha} \ar[r]^{\gamma}  & 2 \ar@(ur,dr)[]^{\beta}
} 
\] 
with relations $\alpha^2=\beta^2=\gamma\beta-\beta\gamma=0$. Then, $A$ is an Iwanaga--Gorenstein algebra with $\injdim A\leq 1.$ Take $Q$ to be the projective cover of the simple $1$. So, $Q$ is projective-injective and $Q\ldom_A A=\domdim A=1$. This means that $(A, Q)$ is a relative $1$-Auslander--Gorenstein pair. However, the correspondence does not work for this case.

Under the assignment of Theorem \ref{maintheorem}, the pair $(A, Q)$ is sent to $(\End_A(Q)^{op}, Q)$ and $\End_A(Q)^{op}$ corresponds to the algebra $B=k[x]/(x^2).$ As right $B$-module $Q$ is isomorphic to $k[x]/(x^2)\oplus k[x]/(x^2).$ Hence, it is a generator, cogenerator, and also a projective-injective module over $B$. In particular $\tau_nQ=\tau_n^-Q=0$ for every $n\geq 1$.  So, $Q$ satisfies all conditions to be a quasi-precluster tilting module, but the problem is that there is no way of recovering $A$ from $k[x]/(x^2).$ Indeed, $\End_B(Q)$ is Morita equivalent to $k[x]/(x^2)$, and hence it is different from $A$.
\end{Example}

In the previous example, the correspondence did not hold because the self-orthogonal module did not possess a double centralizer property. In the following, we see an example that shows that even if we have a double centralizer property, the correspondence might fail if we drop the assumption $n\geq m+l+2$.
\begin{Example}
Consider again the algebra $R$ defined in (\ref{ringeldualstwofour}) which is the Ringel dual of the basic algebra of $S(2, 4)$ and the notation associated with that example. The indecomposable projective modules are 
\[
P_R(1) = \begin{tikzcd}[every arrow/.append style={dash}, column sep={0.4em},
		row sep={0.5em}]
 1 \arrow[d] \\ 2 \arrow[d] \\ 1
	\end{tikzcd},
\qquad
P_R(2)=\begin{tikzcd}[every arrow/.append style={dash}, column sep={0.4em},
		row sep={0.5em}]
 & 2 \arrow[dl] \arrow[d]& \\ 1 & 3 \arrow[d] & \\ & 2 &
	\end{tikzcd}
\qquad
\text{and}
\qquad
P_R(3) = \begin{tikzcd}[every arrow/.append style={dash}, column sep={0.4em},
		row sep={0.5em}]
3 \arrow[d] \\ 2
	\end{tikzcd}.
\]
So, we can see that the characteristic tilting module of $R$, $T^R:=T_R(1)\oplus T_R(2)\oplus T_R(3)$ has projective (resp. injective) dimension exactly two. In particular, $\pdim_{R^{op}} DT^R=\injdim_{R^{op}} DT^R=2,$ and of course, $DT_R\ldom R^{op}=+\infty$. Thanks to $R$ being a quasi-hereditary algebra with a simple preserving duality, it follows that $R^{op}$ has global dimension $4$. So, $(R^{op}, DT^R)$ is a relative $4$-Auslander pair. Hence, $m+l+2=6>n$.
Under the assignment of Theorem \ref{maintheorem}, the pair $(R^{op}, DT^R)$ is sent to $(A^{op}, T\simeq \Hom_A(A, T))$. Hence, $\begin{tikzcd}[every arrow/.append style={dash}, column sep={0.4em},
		row sep={0.5em}]
 5 \arrow[d] \\ 4 \arrow[d] \\ 5
	\end{tikzcd}$ is a direct summand of $T$ as a right $A^{op}$-module and the transpose of this direct summand is the cokernel of the non-zero map $Ae_5\rightarrow Ae_3$, that is, it is the simple $3$, where $e_i$ corresponds to the primitive idempotent associated with the vertex $i$. So $\tau T=\tau_{4-2-1} T$ contains the simple $3$ as direct summand which is not in the top of $T$. This shows that (iv) of Proposition \ref{Prop5dot4} does not hold for $T$ as right $A^{op}$-module. This example also shows that tilting modules are not necessarily quasi-precluster tilting modules in the sense of Definition \ref{defnmlpreculuster}.
\end{Example}

As we see below, it is not reasonable to impose the condition $n\geq m+l+2$ on tilting modules, and so the aim of the next proposition is to clarify that tilting modules and quasi-precluster tilting modules are two different entities.

\begin{Prop}
    Let $\Lambda$ be a finite-dimensional algebra over a field and let $T$ be a right tilting $\Lambda$-module so that $A:=\End_\Lambda(T)$ is Iwanaga-Gorenstein. Then $(A, T)$ is a relative $n$-Auslander--Gorenstein pair with $n\leq \pdim_A T+\injdim_A T=\pdim T_\Lambda+\injdim T_\Lambda$.
\end{Prop}
\begin{proof}
    By the Higher Morita-Tachikawa correspondence,  $T$ being a tilting $\Lambda$-module, implies that $T$ is also a tilting over $A$ (see also for example \cite{zbMATH03925033}).
    Now, $A$ being Iwanaga--Gorenstein implies that $T$ is a tilting and a cotilting module, and therefore also a cotilting module over $\Lambda$. In particular, $\Lambda$ must be Iwanaga--Gorenstein as well (see  \citep[Lemma 1.3]{zbMATH00957373}).
    It is well-known that $\pdim_A T=\pdim T_\Lambda$. For convenience, we provide a short argument. 
    Since $T$ is a tilting $A$-module, it is an $\pdim_A T$-quasi-generator of $\smod{\Lambda}$ by the Higher Morita--Tachikawa correspondence. Hence, $\Ext_\Lambda^{\pdim_A T}(T, A)\neq 0$. Thus, $\pdim_A T\leq \pdim T_\Lambda$. Symmetrically, it follows that $\pdim T_\Lambda\leq \pdim_A T.$ Analogously, $\injdim_A T=\injdim T_\Lambda.$

It remains to show that $\injdim_A A\leq \pdim T_\Lambda+\injdim T_\Lambda.$ 

But this follows immediately from the fact that $T$ is an $\pdim T_\Lambda$-quasi-generator of $A\lsmod$ thanks to the Higher Morita-Tachikawa correspondence, and so $\injdim_A A\leq \pdim T_\Lambda+ \injdim_A T=\pdim T_\Lambda+\injdim T_\Lambda$.
\end{proof}

\section{Cohen--Macaulay modules of Auslander--Gorenstein pairs}
\label{section: CM}

Our aim in this section is to describe the category of Gorenstein projective modules of $A$ for a relative $n$-Auslander--Gorenstein pair $(A, Q)$. Further, we want to characterise relative $n$-Auslander pairs among the $n$-Auslander--Gorenstein pairs using properties of $Q$ as $\End_A(Q)^{op}$-module. To this end, we will follow closely the theory of relative homological algebra using Auslander and Solberg's perspective from \cite{zbMATH00423523, zbMATH00423522}. These techniques have been also used in \cite{IS} and \cite{LiZhang}.

\subsection{Recap on \texorpdfstring{$F$}--exact sequences }

Given $M\in \smod{\Lambda}$, we can consider a new exact structure in $\smod{\Lambda}$ that makes $M$ a projective object and another that makes $M$ an injective object. Indeed, we say that an exact sequence $\xymatrix{
	0 \ar[r]^{} & X \ar[r]^{} & Y \ar[r]^{} & Z \ar[r]^{} &  0}$ is an \textbf{$F_M$-exact} sequence if it is exact under $\Hom_{\Lambda}(M, -)$, and it is said to be \textbf{$F^M$-exact} if it is exact under $\Hom_{\Lambda}(-, M)$. The module category $\smod{\Lambda}$ together with all $F_M$-exact sequences forms an exact category with enough projectives and enough injectives (see for example \citep[Proposition 2.2]{IS} and \citep[Proposition 1.10]{zbMATH00423522}) whose projective objects are the modules in $\add \Lambda\oplus M$. The injective objects of the above exact category are the modules in $\add D\Lambda \oplus \tau M$. It is well-known that an exact sequence is $F_M$-exact if and only if it is $F^{\tau M}$-exact (see \citep[Proposition 1.7]{zbMATH00423522}).
So, it is enough to present the notation to $F_M$-exact sequences. 

Given $X\in \smod{\Lambda}$, an \textbf{$F_M$-projective resolution} of $X$ is any $F_M$-exact sequence $$ \xymatrix{\cdots \ar[r]^{} & P_1 \ar[r]^{} & P_0 \ar[r]^{} &X \ar[r]^{} & 0 }$$ whose terms  $P_i$ belong to $\add \Lambda\oplus M$, while an \textbf{$F_M$-injective resolution} is any $F_M$-exact sequence $ \xymatrix{ 0\ar[r]^{} & X\ar[r]^{} & I_0 \ar[r]^{} &I_1 \ar[r]^{} & \cdots }$
whose terms $I_i$ belong to $\add D\Lambda\oplus \tau M$.

We denote by $\Ext_{F_M}^i(X, Y)$ the $i$-th \textbf{$F_M$-relative extension group}. Hence, $\Ext_{F_M}^i(X, Y)$ is exactly the $i$-th cohomology $H^i(\Hom_{\Lambda}(X^\bullet, Y))$, where $X^\bullet$ is a deleted $F_M$-projective resolution of $X$.
As in the classical case, $\Ext_{F_M}^i(X, Y)$ can be determined using deleted $F_M$-injective resolutions of $Y$.

Using $\Ext_{F_M}^i(-, -)$, we can analogously to the classical case, define the subcategories $ {}^{\perp_{F_M}} X$ and $X^{\perp_{F_M}}$ as well as the concepts of $F_M$-projective dimension of $X$ and $F_M$-injective dimension of $X$ which we denote by $\pdim_{F_M} X$ and $\injdim_{F_M} X$, respectively. 

\begin{Def}
	Let $M\in \smod{\Lambda}$. A module $T$ in $\smod{\Lambda}$ is an {\bf $F_M$-cotilting object} if the following conditions are satisfied: 
	\begin{enumerate}[\normalfont(i)]
		\item $\injdim_{F_M} T<+\infty$;
		\item $\Ext_{F_M}^{i>0}(T, T)=0$;
		\item there exists an $F_M$-exact sequence 
		$$\xymatrix{0 \ar[r]^{} & T_n\ar[r]^{} & \cdots \ar[r]^{} & T_1 \ar[r]^{} & T_0 \ar[r]^{} & D\Lambda \oplus \tau M\ar[r]^{} & 0 },$$ with $T_i\in \add T$ for some non-negative integer $n$.
	\end{enumerate}
\end{Def}

\subsection{Relative cotilting objects induced by Auslander--Gorenstein pairs} We are now ready to show that $Q$ is a relative cotilting object in the exact structure making a certain syzygy of $Q$ a projective object over $\End_A(Q)^{op}$ for a relative Auslander--Gorenstein pair $(A, Q)$.

\begin{Prop}
\label{Prop8dot1}
Let $n, m, l$ be natural numbers so that $n\geq m+l+2$. Let $(A, Q)$ be a relative $n$-Auslander--Gorenstein pair with $\pdim_A Q=l$, $\injdim_A Q=m$ and $_{A}Q\in Q^\perp$. Write $\Lambda=\End_A(Q)^{op}$. Then, $Q$ is an $F_{\Omega^{n-m-2}(Q)}$-cotilting object.
\end{Prop}
\begin{proof}
Observe that the modules in $\add D\Lambda \oplus \tau_{n-m-1}Q$ are exactly the injective objects of the exact structure whose exact sequences remain exact under $\Hom_\Lambda(\Omega^{n-m-2}(Q), -)$. So we will start by showing the existence of a finite relative $F_{\Omega^{n-m-2}(Q)}$-resolution of $ D\Lambda \oplus \tau_{n-m-1}Q$ by terms in $\add Q$. By Lemma \ref{lemma5dot3} and the fact that $Q$ is an $m$-quasi-cogenerator of $\Lambda\lsmod$, there are exact sequences
	\begin{equation}
		\xymatrix@R=0.1cm{0 \ar[r] & Q_m'' \ar[r] & \cdots \ar[r] & Q_0'' \ar[r]& \tau_{n-m-1}Q\ar[r] & 0\\
		0 \ar[r] & Q_m'\ar[r] & \cdots \ar[r]& Q_0' \ar[r] & D\Lambda \ar[r] & 0
	}
	\end{equation} which remain exact under $\Hom_{\Lambda}(Q, -)$ with $Q_i'', Q_i'\in \add Q_\Lambda$. So, in particular, there exists an exact sequence
	\begin{equation}
	\xymatrix@R=0.1cm{0 \ar[r] & Q_m \ar[r] & \cdots \ar[r] & Q_0 \ar[r]& \tau_{n-m-1}Q\oplus D\Lambda\ar[r] & 0
	} \label{eq32}
\end{equation}which remains exact under $\Hom_\Lambda(Q, -)$ with $Q_i\in \add Q$. Define inductively $W_0:=D\Lambda\oplus \tau_{n-m-1}Q$ and $W_{i+1}$ as the kernel of $Q_i\rightarrow W_i$ for $i=0, \ldots, m-1$. If $m=0$, then it is clear that (\ref{eq32}) remains exact under $\Hom_{\Lambda}(\Omega^{n-m-2}(Q), -)$ (since (\ref{eq32}) is just an isomorphism in such a case). Assume that $m\geq 1$.
Observe that $\Ext_\Lambda^i(\Omega^{n-m-2}(Q), Q)\cong \Ext_{\Lambda}^{i+n-m-2}(Q, Q)=0$ for $i=1, \ldots, m$. Then, given $i\in\{1, \ldots, m\}$ we obtain, by construction, 
\begin{align*}
	\Ext_{\Lambda}^j(\Omega^{n-m-2}(Q), W_i)&\cong \Ext_\Lambda^{j+1}(\Omega^{n-m-2}(Q), W_{i+1})\\ &\cong \Ext_{\Lambda}^{j+m-i}(\Omega^{n-m-2}(Q), W_m)\\ &=\Ext_{\Lambda}^{j+m-i}(\Omega^{n-m-2}(Q), Q_m)=0,
\end{align*}
for $j=1, \ldots, i$. In particular, $\Ext_{\Lambda}^1(\Omega^{n-m-2}(Q), W_i)=0$ for every $i=1, \ldots, m$ and so (\ref{eq32}) remains exact under $\Hom_{\Lambda}(\Omega^{n-m-2}(Q), -)$, that is, it is an $F_{\Omega^{n-m-2}(Q)}$-exact sequence.

We will now proceed to show that the relative $F_{\Omega^{n-m-2}(Q)}$-injective dimension of $Q$ is finite. We infer by Remark~\ref{Rmk7dot1} and \cite[Corollary~6.5]{CrPs}
that 
\begin{align}
	\add \Omega^{-m}(Q) \oplus D\Lambda &=\add D(\Omega^m(DQ)\oplus \Lambda) \nonumber \\ &= \add D \Hom_A(Q, X_{n-m})\oplus D\Lambda \nonumber \\ & = \add \tau_{n-m-1}Q\oplus D\Lambda, \label{eq33}
\end{align} 
where $X_{n-m}$ fits into an exact sequence $0\rightarrow A\rightarrow Q_0\rightarrow\cdots\rightarrow Q_{n-m-1}\rightarrow X_{n-m}\rightarrow 0$ which is sent through $\Hom_A(-, Q)$ to the beginning of the minimal projective resolution of $Q$ (and $m$ is not necessarily non-zero). So, the minimal injective resolution of $Q$ as $\Lambda$-module \begin{equation}
	\xymatrix{0 \ar[r] & Q \ar[r] & I_0\ar[r] & \cdots \ar[r] & I_{m-1} \ar[r] & \Omega^{-m}(Q)\ar[r] & 0} \label{eqd33}
\end{equation} gives a coresolution of $Q$ into $F_{\Omega^{n-m-2}(Q)}$-injective objects. Since $\Ext_{\Lambda}^i(\Omega^{n-m-2}(Q), Q)=0$ for $i=1, \ldots, m$ we obtain that (\ref{eqd33}) remains exact under $\Hom_{\Lambda}(\Omega^{n-m-2}(Q), -)$. This means that (\ref{eqd33}) is an injective $F_{\Omega^{n-m-2}(Q)}$-exact coresolution and so $\injdim_{F_{\Omega^{n-m-2}(Q)}} Q\leq m$. Thus, $\Ext_{F_{\Omega^{n-m-2}(Q)}}^i(Q, Q)=0$ for $i>m$. On the other hand, $\Ext_{F_{\Omega^{n-m-2}(Q)}}^i(Q, Q)\subset \Ext_\Lambda^i(Q, Q)=0$ for $i=1, \ldots, m$, and thus $\Ext_{F_{\Omega^{n-m-2}(Q)}}^i(Q, Q)=0$ for every $i>0$. This implies that $Q$ is an $F_{\Omega^{n-m-2}(Q)}$-cotilting object.
\end{proof}

\begin{Remark}
\label{remarkm=0}
	It follows by the proof of Proposition \ref{Prop8dot1} that $Q$ is an $F_{\Omega^{n-m-2}(Q)}$-injective object if $m=0$. It turns out that these two statements are equivalent. In fact, assume that $Q$ is an $F_{\Omega^{n-m-2}(Q)}$-injective object. Then $D\Hom_A(Q, DA)\cong Q\in \add D\Hom_A(Q, X_{n-m} \oplus Q)$ using the notation of the proof of Proposition \ref{Prop8dot1}. Since $Q\lcodom_A Q\oplus X_{n-m}\geq 2$ and $Q\lcodom_A DA\geq 2$ this implies that $DA\in \add Q\oplus X_{n-m}$. But $Q\oplus X_{n-m}$ is a tilting-cotilting $A$-module and so the number of direct summands of $DA$ and of $Q\oplus X_{n-m}$ must coincide, that is, $\add DA=\add Q\oplus X_{n-m}\ni Q$. Thus, $m=0$.
\end{Remark}

We continue by proving that $Q$ is a relative tilting object in an exact category where a certain cosysygy of $Q$ is an injective object.

\begin{Prop}\label{Prop8dot3}
Let $n, m, l$ be natural numbers so that $n\geq m+l+2$. Let $(A, Q)$ be a relative $n$-Auslander--Gorenstein pair with $\pdim_A Q=l$, $\injdim_A Q=m$ and $_{A}Q\in Q^\perp$. Write $\Lambda=\End_A(Q)^{op}$. Then, $Q$ is an $F^{\Omega^{-(n-l-2)}(Q)}$-tilting object.
\end{Prop}
\begin{proof}
From \cite[Remark~4.9]{CrPs} we have that $(A^{op}, DQ)$ is a relative $n$-Auslander--Gorenstein pair with $l=\injdim_{A^{op}} DQ$, $m=\pdim_{A^{op}} DQ$ and $n\geq m+l+2$. By Proposition \ref{Prop8dot1}, $DQ$ is an $F_{\Omega^{n-l-2}(DQ)}$-cotilting object. Since $D$ is an exact functor and $\Hom_{\Lambda}(M, X)\cong \Hom_{\Lambda}(DX, DM)$ for every $M, X\in \Lambda\lsmod$, a sequence $\delta$ is $F_M$-exact if and only if $D\delta$ is $F^{DM}$-exact. Recall that $D\Omega^{n-l-2}(DQ)\cong \Omega^{-(n-l-2)}(Q)$. Hence, for every $i>0$ we have
\[
\Ext_{F^{\Omega^{-(n-l-2)}(Q)}}^i(Q, Q) \cong \Ext_{F_{\Omega^{n-l-2}(DQ)}}^i(DQ, DQ)=0.
\]
Now since the injective objects of $F_{\Omega^{n-l-2}(DQ)}$ are $\add D\Lambda \oplus \tau_{n-l-1}DQ$ and the projective objects of $F^{\Omega^{-(n-l-2)}(Q)}$ are $\add (\Lambda \oplus \tau_{n-l-1}^-Q)$ we obtain that $D$ interchanges the projective objects of $F^{\Omega^{-(n-l-2)}(Q)}$ with the injective objects of $F_{\Omega^{n-l-2}(DQ)}$. We conclude that $Q$ is an  $F^{\Omega^{-(n-l-2)}(Q)}$-tilting object.
\end{proof}

\begin{Cor}
	Let $(A, Q)$ be a relative $n$-Auslander--Gorenstein pair with $_{A}Q\in Q^\perp$ and $n= m+l+2$, where $m:=\injdim_{A} Q$ and $l:=\pdim_{A} Q$. Write $\Lambda=\End_A(Q)^{op}$. Then, $Q$ is an $F_{\Omega^{l}(Q)}$-tilting-cotilting object.
\end{Cor}
\begin{proof}
It follows by \cite[Corollary~6.5]{CrPs} that $\add \Lambda\oplus \Omega^{n-m-2}(Q)=\add \Lambda \oplus \tau_{n-l-1}^-Q$ since $n-l=m+2=n-(n-m-2)$. So the exact structures induced by $F^{\Omega^{-(n-l-2)}(Q)}$ and $F_{\Omega^{n-m-2}(Q)}=F_{\Omega^{l}(Q)}$ coincide. The result then follows by Propositions \ref{Prop8dot1} and \ref{Prop8dot3}. 
\end{proof}

Given a positive integer $n$ %\in \mathbb{N}$ 
and $M\in \smod{A}$, we write  
$${}^{\perp_n}  M:=\{X\in \smod{\Lambda}\ | \  \Ext_{\Lambda}^i(X, M)=0, \ i=1, \ldots, n \}$$ 
and 
$$
M^{\perp_n}:=\{X\in \smod{\Lambda}\ | \ \Ext_{\Lambda}^i(M, X)=0, \ i=1, \ldots, n \}.
$$
Recall that over an Iwanaga--Gorenstein algebra $A$, the category $\CM(A)$ of Cohen-Macaulay $A$-modules are the modules in $\lerp A$. We are now ready to state and prove the main theorem of this section.

\begin{Theorem}
\label{thm8dot5}
Let $(A, Q)$ be a relative $n$-Auslander--Gorenstein pair with $_{A}Q\in Q^\perp$ and $n\geq m+l+2$, where $m:=\injdim_{A} Q$ and $l:=\pdim_{A} Q$. Write $\Lambda=\End_A(Q)^{op}$. The following assertions hold.
\begin{enumerate}[\normalfont(i)]
\item Then, $\Hom_\Lambda(-, Q)$ induces an exact duality between ${}^{\perp_m}  Q\cap Q^{\perp_{n-m-2}}$ and $\lerp {_A A} $.

\item Then, $\Hom_{\Lambda}(Q, -)$ induces an exact equivalence of categories between ${}^{\perp_{n-l-2}}  Q\cap Q^{\perp_{l}}$ and $\lerp A_A$.

\item There is a triangle equivalence$\colon$ 
\[
\xymatrix{
{}^{\perp_m}  \underline{Q\cap Q}^{\perp_{n-m-2}} \ar[r]^{ \ \ \  \simeq} &  \underline{\CM}(A^{\op})},
\]
where ${}^{\perp_m}  \underline{Q\cap Q}^{\perp_{n-m-2}}$ is the stable category of ${}^{\perp_m}  Q\cap Q^{\perp_{n-m-2}}$. %modulo the ideal generated by $\add{Q}$.

\item There is a triangle equivalence$\colon$ 
\[
\xymatrix{
{}^{\perp_{n-l-2}}  \underline{Q\cap Q}^{\perp_{l}} \ar[r]^{ \ \ \ \simeq} &  \underline{\CM}(A)},
\]
where ${}^{\perp_{n-l-2}}  \underline{Q\cap Q}^{\perp_{l}}$ is the stable category of ${}^{\perp_{n-l-2}}  Q\cap Q^{\perp_{l}}$. %modulo the ideal generated by $\add{Q}$.

\end{enumerate}
\end{Theorem}
\begin{proof}
(i)	Observe that the $F_{\Omega^{n-m-2}(Q)}$-exact sequences are precisely the $F_{\Lambda\oplus \Omega^{n-m-2}(Q)}$-exact sequences since $\Lambda$ is of course a projective $\Lambda$-module. To simplify the notation, we write $F:=F_{\Omega^{n-m-2}(Q)}$ and ${}^{\perp_F} Q$ to denote ${\{X\in \smod{\Lambda}\ | \ \Ext_F^i(X, Q)=0, \forall i>0 \}}$. Similarly, we denote by $Q^{\perp_F}$ the subcategory ${\{X\in \smod{\Lambda}\ | \ \Ext_F^i(Q, X)=0, \forall i>0 \}}$. 
	
	By Proposition \ref{Prop8dot1}, $Q$ is an $F$-cotilting object. By \citep[Theorem 3.2(a), Corollary 3.6(a), Proposition 3.8(b)]{zbMATH00423523}, $\Hom_{\Lambda}(-, Q)$ restricts to a duality ${}^{\perp_F} Q\rightarrow {}^\perp C$ with $$C=\Hom_{\Lambda}(\Lambda\oplus \Omega^{n-m-2}(Q), Q)\cong Q\oplus \Hom_\Lambda(\Omega^{n-m-2}(Q), Q)\in A\lsmod.$$
	Since $Q\ldom_A A\geq n\geq n-m-2$, there exists an exact sequence 
	\begin{equation}
		\xymatrix{0 \ar[r] & A \ar[r] & Q_0 \ar[r] & \cdots \ar[r] & Q_{n-m-3} \ar[r] & X_{n-m-2} \ar[r] & 0}\label{eq34}
	\end{equation} which remains exact under $\Hom_A(-, Q)$ with $Q_i\in \add Q$ such that $\Omega^{n-m-2}(Q)\cong \Hom_A(X_{n-m-2}, Q)$. Thanks to $n-(n-m-2)=m+2\geq 2$ we obtain that $Q\ldom_A X_{n-m-2}\geq 2$. Hence, $\Hom_{\Lambda}(\Hom_A(X_{n-m-2}, Q), Q)\cong X_{n-m-2}$. Therefore $C\cong Q\oplus X_{n-m-2}$.
	Using the exact sequence (\ref{eq34}) it follows that $X_{n-m-2}$ has finite projective dimension, and so $C$ has finite projective dimension. Let $N\in \lerp {}_A A$. Then $N\in \lerp C$.
	To see this, let 	\begin{equation}
		\xymatrix{0 \ar[r] & P_t \ar[r] & \cdots \ar[r] & P_0 \ar[r] & C \ar[r] & 0 }
	\end{equation} be a projective resolution of $C$. As $N\in \lerp {}_A A$ we get $\Ext_A^i(N, P_j)=0$ for $i>0$. Hence, by applying $\Hom_A(N, -)$ we infer that $\Ext_A^i(N, C)=0$ for $i>0$ using dimension shifting. So, there exists an object $L\in {}^{\perp_F} Q$ so that $N=\Hom_{\Lambda}(L, Q)$. By \citep[Theorem 3.2(a), Proposition 3.7]{zbMATH00423523}, we have
 $$\Ext_F^i(Q, L)\cong \Ext_A^i(\Hom_{\Lambda}(L, Q), \Hom_{\Lambda}(Q, Q))\cong \Ext_A^i(N, A)=0, \ \forall i>0.$$
Thus, $L\in {}^{\perp_F} Q \cap Q^{\perp_F}$. So, it follows that $\Hom_{\Lambda}(-, Q)$ restricts to a duality $ {}^{\perp_F} Q \cap Q^{\perp_F}\rightarrow \lerp {}_A A$.

We will now proceed to simplify ${}^{\perp_F} Q$ and $Q^{\perp_F}$.

By \citep[Proposition 2.5(b)]{IS}, we get $\Ext_F^i(Y, Q)\cong \Ext_\Lambda^i(Y, Q)$ for every $i=1, \ldots, m$ and $Y\in \smod{\Lambda}$ because $\Ext_\Lambda^i(\Omega^{n-m-2}(Q), Q)=0$ for $i=1, \ldots, m$.
Hence, 
\begin{align}
	{}^{\perp_F} Q &=\{X\in \smod{\Lambda}\ | \ \Ext_F^{i>0}(X, Q)=0\} \nonumber \\
      &= \{X\in \smod{\Lambda}\ | \ \Ext_F^i(X, Q)=0, \ i=1, \ldots, m \} \nonumber \\
      &=\{X\in \smod{\Lambda}\ | \ \Ext_\Lambda^i(X, Q)=0, \ i=1, \ldots, m \} \nonumber \\
      &= {}^{\perp_m}  Q,
\end{align}where the second equality follows from $\injdim_F Q\leq m$.

By \cite[Corollary~6.5]{CrPs}, we obtain that $\add D\Lambda \oplus \tau_{n-m-1}Q=\add D\Lambda \oplus \Omega^{-m}(Q)$, see for instance
(\ref{eq33}). Hence, $F=F_{\Omega^{n-m-2}(Q)}=F^{\tau_{n-m-1}(Q)}=F^{\Omega^{-m}(Q)}$. So the beginning of the minimal projective resolution of $Q$, \begin{equation}
	\xymatrix{0 \ar[r] & \Omega^{n-m-2}(Q) \ar[r] & P_{n-m-3}\ar[r] & \cdots \ar[r] & P_0\ar[r]& Q\ar[r] & 0},
\end{equation} is $F$-exact because $\Ext_\Lambda^1(\Omega^{j}(Q), \Omega^{-m}(Q))=\Ext_\Lambda^{1+j+m}(Q, Q)=0$ for $j=1, \ldots, n-m-3$. As the $F$-projective objects are the modules in $\add \Lambda\oplus \Omega^{n-m-2}(Q)$, it follows that $\pdim_F Q\leq n-m-2$. Now using \citep[Proposition 2.5(a)]{IS} together with the fact that $\Ext_\Lambda^i(Q, \Omega^{-m}(Q))=\Ext_\Lambda^{i+m}(Q, Q)=0$ for $i=1, \ldots, n-m-2$, we infer that $\Ext_F^i(Q, X)\cong \Ext_\Lambda^i(Q, X)$ for $i=1, \ldots, n-m-2$. Thus, 
\begin{align*}
	Q^{\perp_F}&=\{X\in \smod{\Lambda}\ | \ \Ext_F^{i>0}(Q, X)=0\} \\ &= \{X\in \smod{\Lambda}\ | \ \Ext_F^i(Q, X)=0, \  i=1, \ldots, n-m-2\}\\ &=\{X\in \smod{\Lambda}\ | \ \Ext_\Lambda^i(Q, X)=0, \  i=1, \ldots, n-m-2  \} \\ &=Q^{\perp_{n-m-2}}.
\end{align*} 
It remains to show that the duality between ${}^{\perp_m}  Q\cap Q^{\perp_{n-m-2}}$ and $\lerp {_A A}$ is exact. Consider an $F$-exact sequence $0\to X\to Y\to Z\to 0$ in  ${}^{\perp_m}  Q\cap Q^{\perp_{n-m-2}}$. Applying $\Hom_{\Lambda}(-,Q)$ we obtain the exact sequence $0\to \Hom_{\Lambda}(Z,Q)\to \Hom_{\Lambda}(Y,Q)\to \Hom_{\Lambda}(X,Q) \to \Ext^1_{\Lambda}(Z,Q)$. 
If $m=0$, then Remark~\ref{remarkm=0} yields that $Q$ is an $F_{\Omega^{n-m-2}(Q)}$-injective object and 
 therefore the above $F$-exact sequence remains exact under $\Hom_\Lambda(-, Q)$.
Assume now that $m\geq 1$. Since $Z$ lies in ${}^{\perp_m}Q$, we infer that $\Ext^1_{\Lambda}(Z,Q)=0$ and therefore the exactness of the duality follows. %Thus, (i) holds.
So, this completes the proof of (i).

(ii) Since $(A^{op}, DQ)$ is a relative $n$-Auslander--Gorenstein pair with $n\geq m+l+2$, where $m=\pdim_{A^{op}} DQ$ and $l=\injdim_{A^{op}} DQ$, statement (i) yields that the functor $\Hom_\Lambda(-, DQ)$ restricts to a duality $ {}^{\perp_l}  DQ\cap DQ^{\perp_{n-l-2}}\rightarrow \lerp A_A$. Observe that $\Hom_\Lambda(-, DQ)\cong \Hom_\Lambda(Q, D-)\cong \Hom_{\Lambda}(Q, -)\circ D$. Since $D$ induces a duality between $ {}^{\perp_{n-l-2}}  Q\cap Q^{\perp_{l}}$ and $ {}^{\perp_l}  DQ\cap DQ^{\perp_{n-l-2}}$ we conclude that $\Hom_\Lambda(Q, -)$ restricts to an equivalence of categories ${}^{\perp_{n-l-2}}  Q\cap Q^{\perp_{l}}\rightarrow \lerp A_A$.

(iii)  Clearly, the category ${}^{\perp_m}{Q\cap Q}^{\perp_{n-m-2}}$ is closed under extensions. Hence, ${}^{\perp_m}{Q\cap Q}^{\perp_{n-m-2}}$ is an exact category (in the sense of Quillen). Moreover, the object $Q$ is injective in ${}^{\perp_m}{Q}$ and it is projective in ${Q}^{\perp_{n-m-2}}$. We infer that the object $Q$ is projective-injective in the intersection ${}^{\perp_m}{Q\cap Q}^{\perp_{n-m-2}}$. Using the duality of (i), it follows that the modules in $\add{Q}$ are precisely the injective-projective objects of ${}^{\perp_m}{Q\cap Q}^{\perp_{n-m-2}}$. Moreover, from the fact that in (i) we have an exact duality, we get that ${}^{\perp_m}{Q\cap Q}^{\perp_{n-m-2}}$ has enough projectives. We infer that ${}^{\perp_m}{Q\cap Q}^{\perp_{n-m-2}}$ is an exact Frobenius category and therefore the stable category ${}^{\perp_m}  \underline{Q\cap Q}^{\perp_{n-m-2}}$ is triangulated \cite{Happel}. Recall that the objects of ${}^{\perp_m}  \underline{Q\cap Q}^{\perp_{n-m-2}}$ are the objects of ${}^{\perp_m}{Q\cap Q}^{\perp_{n-m-2}}$ and the morphism space consists of morphisms in ${}^{\perp_m}{Q\cap Q}^{\perp_{n-m-2}}$ modulo the subgroup consisting of the maps factoring through an object in $\add{Q}$. The desired triangle equivalence between ${}^{\perp_m}  \underline{Q\cap Q}^{\perp_{n-m-2}}$ and $\underline{\CM}(A^{\op})$
follows immediately by (i). Similarly using (ii), we can prove the triangle equivalence in (iv).
\end{proof}

\begin{Cor}
		Let $(A, Q)$ be a relative $n$-Auslander--Gorenstein pair with $_{A}Q\in Q^\perp$ and $n\geq m+l+2$, where $m:=\injdim_{A} Q$ and $l:=\pdim_{A} Q$. Write $\Lambda=\End_A(Q)^{op}$. Then, the following assertions are equivalent:
		\begin{enumerate}[\normalfont(i)]
			\item $(A, Q)$ is a relative $n$-Auslander pair;
			\item $\add Q_\Lambda= {}^{\perp_m}  Q\cap Q^{\perp_{n-m-2}}$;
			\item $\add Q_\Lambda={}^{\perp_{n-l-2}}  Q\cap Q^{\perp_{l}}$.
		\end{enumerate}
\end{Cor}
\begin{proof}
Since $A$ is an Iwanaga-Gorenstein algebra, it has finite global dimension if and only if the modules in $\lerp {}_A A$ are precisely the projective $A$-modules. By projectivization, $\Hom_{\Lambda}(-, Q)$ provides a duality between $\add Q$ and $\add {}_A A$. By Theorem \ref{thm8dot5}(i), the equivalence between (i) and (ii) follows. Analogously, using Theorem \ref{thm8dot5}(ii), the equivalence between (i) and (iii) holds.
\end{proof}

\begin{Example}
    We continue now with the relative Auslander pair $(A, Q)=(S(2, 4), V^{\otimes 4}).$ As we have seen in Section \ref{schuralgebra}, in this case the parameters are $n=4$ and $l=m=1$. So, over the algebra $\End_A(Q)^{op}$ we have the following identification
$$\add Q=\add  \left( 1 \oplus \begin{tikzcd}[every arrow/.append style={dash}, column sep={0.4em},
		row sep={0.5em}]
 1 \arrow[d] \\ 2 \arrow[d] \\1
	\end{tikzcd} \oplus  \begin{tikzcd}[every arrow/.append style={dash}, column sep={0.4em},
		row sep={0.5em}]
  & 2 \arrow[dl] \arrow[d] & 1 \arrow[dl]\\
		1 & 2 &  
	\end{tikzcd} \right)= {}^{\perp_1} Q\cap Q^{\perp_1}$$
    However, ${}^{\perp_1} Q\not\subset \add Q $ since for example the projective cover of $2$ does not belong to $\add Q$. So, $Q$ cannot be a classical cluster tilting object.
    Observe also that $Q$ fails to be an (1, 1)-ortho-symmetric module in the sense of \cite{ChenKoenig} only because $Q$ fails to be a generator-cogenerator.
\end{Example}

\section*{Acknowledgments}
This work began when the first-named author was visiting the Max Planck Institute for Mathematics at Bonn, and so the first-named author would like to thank the MPIM for its hospitality during the stay. For the second-named author, the research project is implemented in the framework of H.F.R.I call ``Basic research Financing (Horizontal support of all Sciences)" under the National Recovery and Resilience Plan ``Greece 2.0" funded by the European Union -- NextGenerationEU (H.F.R.I. Project Number: 76590).

\bibliographystyle{alphaurl}
\bibliography{ref}

\end{document}